\documentclass[final,3p,times]{elsarticle}

\usepackage{amssymb}
\usepackage{amsmath}
\usepackage[utf8]{inputenc}
\usepackage[T1]{fontenc}
\usepackage{amsmath, amsthm, mathtools, amssymb}
\usepackage{algorithm}
\usepackage{algpseudocode}
\usepackage{mathrsfs}
\usepackage{makecell}
\usepackage{shuffle}
\usepackage{bbm}
\usepackage{subcaption}
\usepackage[hidelinks]{hyperref}

\usepackage{csquotes}
\usepackage{nomencl}
\usepackage{enumitem}
\usepackage{xcolor}

\usepackage{amsmath}
\usepackage{mathtools}
\usepackage{amsfonts}
\usepackage{amsthm}

\newtheorem{theorem}{Theorem}[section]
\newtheorem{remark}[theorem]{Remark}
\newtheorem{proposition}[theorem]{Proposition}
\newtheorem{lemma}[theorem]{Lemma}
\newtheorem{corollary}[theorem]{Corollary}
\newtheorem{assumption}[theorem]{Assumption}

\newtheorem{example}[theorem]{Example}
\usepackage{graphicx}

\newcommand{\norm}[1]{\lVert#1\rVert}

\newcommand{\scalar}[1]{\langle #1 \rangle}

\makeatletter
\renewcommand\subsection{\@startsection{subsection}{2}{\z@}%
    {20\p@ \@plus 6\p@ \@minus 3\p@}%
    {10\p@ \@plus 3\p@ \@minus 2\p@}%
    {\normalfont\normalsize\itshape}}
\makeatother

\numberwithin{equation}{section}

\journal{Stochastic Processes and their Applications}

\begin{document}

\begin{frontmatter}



\title{Hölder regularity for backward stochastic Volterra integral equations and applications to numerical schemes}

\author[inst1]{Pere Diaz-Lozano\corref{cor1}\fnref{funding}}
\cortext[cor1]{Corresponding author. Email address: \texttt{peredl@math.uio.no}}

\author[inst1,inst2]{Giulia Di Nunno\fnref{funding}}

\fntext[funding]{This work was carried out within the SURE-AI Centre, funded by the Research Council of Norway under grant No.~357482.}

\affiliation[inst1]{
    organization={Department of Mathematics, University of Oslo},
    addressline={P.O. Box 1053 Blindern},
    city={Oslo},
    postcode={0316},
    country={Norway}
}

\affiliation[inst2]{
    organization={Department of Business and Management Science, NHH Norwegian School of Economics},
    addressline={Helleveien 30},
    city={Bergen},
    postcode={5045},
    country={Norway}
}

\begin{abstract}
We prove a Hölder-type regularity estimate for the martingale integrand of a backward stochastic Volterra integral equation (BSVIE). The estimate is formulated in $L^p(\Omega)$ after averaging in $L^2$ over the first time variable, and gives an averaged Hölder estimate of order $1/2$ in the second time variable. Our approach is based on the approximation of the BSVIE by a system of BSDEs. For this system, we establish a uniform regularity estimate for the martingale component using Malliavin calculus, and then pass to the limit to obtain the result for the BSVIE. We allow for a general Malliavin differentiable free term and generator. In particular, neither is assumed to come from a forward stochastic differential equation or to have a specific functional form. We also propose an explicit Euler scheme for the approximating BSDE system and show that the regularity estimate yields a convergence rate for the resulting discrete approximation of the BSVIE.
\end{abstract}


\end{frontmatter}

\section{Introduction}

In this paper, we study regularity properties of the martingale integrand in backward stochastic Volterra integral equations (BSVIEs) of the form
\begin{flalign}\label{eq:bsvie-intro}
    Y(t)
    =
    \Psi(t)
    +
    \int_t^T G\big(t,s,Y(s),Z(t,s)\big)\,ds
    -
    \int_t^T Z(t,s)\,dB(s),
    \qquad t\in[0,T].
\end{flalign}
Here, $B$ is a Brownian motion defined on a complete probability space $(\Omega,\mathcal F,\mathbb P)$, and $\mathbb F=(\mathcal F_t)_{t\in[0,T]}$ is the usual augmentation of its natural filtration. The free term $\Psi$ is a process such that $\Psi(t)$ is $\mathcal F_T$-measurable for each $t\in[0,T]$, and need not be adapted. The generator $G$ may be random. The pair $(\Psi,G)$ will be referred to as the data of the BSVIE. Equations of this type are referred to as Type-I BSVIEs in the literature; see \cite{YONG2006779}.

BSVIEs may be viewed as an extension of backward stochastic differential equations (BSDEs). Indeed, a BSDE has the form
\begin{flalign}\label{eq:bsde-intro}
    Y(t)
    =
    \xi
    +
    \int_t^T f(s,Y(s),Z(s))\,ds
    -
    \int_t^T Z(s)\,dB(s),
    \qquad t\in[0,T],
\end{flalign}
where the terminal condition $\xi$ is a fixed $\mathcal{F}_T$-measurable random variable and the generator depends only on the running time variable. In contrast, in a BSVIE both the free term and the generator depend on the initial time $t$, and the martingale integrand becomes a two-parameter object $Z(t,s)$.

BSDEs were introduced by \cite{BismutJean_Michel1973Ccfi} in the case of a linear generator, and were later generalized by \cite{PARDOUX199055} to the nonlinear case. Since then, they have found applications in stochastic control and mathematical finance. In stochastic control, BSDEs arise as adjoint equations in stochastic maximum principles and are closely related to the value function; see, for example, \cite{Peng1990,Pham2009,Touzi2013}. They are also connected to semi-linear parabolic partial differential equations through nonlinear Feynman--Kac formulas; see \cite{PardouxPeng1992}. In finance, BSDEs are used for the pricing and hedging of financial products \cite{ElKarouiN.1997BSDE}, and also provide a natural framework for nonlinear expectations and dynamic risk measures \cite{Peng1997,RosazzaGianinEmanuela2006Rmvg}.

In contrast, BSVIEs of the form \eqref{eq:bsvie-intro} were first introduced in \cite{Lin21032002}, in the case where the free term was constant in time, namely $\Psi(t)\equiv \xi$ for some $\mathcal F_T$-measurable random variable $\xi$. The theory was later extended in \cite{YONG2006779} to allow for genuinely time-dependent free terms, as well as more general Type-II equations, in which the generator may also depend on the additional component $Z(s,t)$. Further developments of the theory can be found, for instance, in \cite{YongJiongmin2008Waro,WANG20194926,HamaguchiYushi2023AfaM}. Since the Type-I equation \eqref{eq:bsvie-intro} is obtained from the Type-II formulation by removing the dependence on $Z(s,t)$, many results developed for Type-II BSVIEs apply directly to the setting considered here.

BSVIEs arise naturally in several contexts where the standard BSDE framework is not sufficiently flexible to capture memory effects or time-inconsistency. One example is provided by stochastic control problems for systems governed by stochastic Volterra integral equations. In such problems, the state dynamics depend on the past, and the corresponding adjoint equations in the Pontryagin maximum principle are BSVIEs rather than BSDEs; see, for instance, \cite{YongJiongmin2008Waro,yufeng_wang_yong_2015,wang_zhang_2018}. They also appear in mathematical finance and economics, for example in the modelling of dynamic risk measures and recursive utilities with non-exponential discounting or other sources of time-inconsistency; see \cite{Yong01112007,di_nunno_2024}. Further connections have been established with time-inconsistent stochastic control problems and equilibrium Hamilton--Jacobi--Bellman equations \cite{wang_yong_2021,Hamaguchi_2021}, as well as with nonlinear Feynman--Kac formulas for Volterra-type equations \cite{WANG20194926,wang_yong_zhang_2022}.

Turning to the numerical approximation of such equations, the BSDE case is by now well developed, with a large literature devoted to Euler-type discretizations of \eqref{eq:bsde-intro}; see, for instance, \cite{zhang_bsde_method,BouchardBruno2004DaaM}. By comparison, numerical results for BSVIEs remain relatively limited. An early contribution is \cite{BenderPokalyuk2013}, where the authors study a class of BSVIEs whose free term is a functional of the Brownian motion and whose generator is independent of $Z$. They obtain a weak approximation by means of discrete BSVIEs driven by a binary random walk.

In a different direction, several works have used systems of BSDEs to approximate BSVIEs of the form \eqref{eq:bsvie-intro}. Given a partition $\pi=\{t_0,\dots,t_N\}$ of $[0,T]$, the BSVIE is approximated by a finite-dimensional system of BSDEs
\begin{flalign*}
    \big(Y^\pi(t_k,\cdot),Z^\pi(t_k,\cdot)\big),
    \qquad 0\leq k\leq N-1,
\end{flalign*}
where each pair is defined on $[t_k,T]$. As the mesh size $|\pi|$ tends to zero, this BSDE system converges to the solution of the BSVIE in the mean-square sense. In \cite{WANG20194926}, this approximation is used to derive a PDE representation of the solution for BSVIEs of the form
\begin{flalign}
    Y(t)
    =
   \psi(t,X(t),X(T))
    + \int_t^T
    G\big(t,s,X(t),X(s),Y(s),Z(t,s)\big)\,ds
    -
    \int_t^T Z(t,s)\,dB(s),
    \label{eq:forw-bckw-BSVIE}
\end{flalign}
where $G$ is deterministic and $X$ solves
\begin{flalign}\label{eq:forward-SDE-intro}
    X(t)
    =
    x
    +
    \int_0^t b(s,X(s))\,ds
    +
    \int_0^t \sigma(s,X(s))\,dB(s).
\end{flalign}
More recently, \cite{HamaguchiYushi2023AfaM} extended this strategy to Type-II BSVIEs. By introducing an Euler-type scheme for the associated BSDE system, the authors obtained a numerical approximation of the BSVIE, together with a convergence rate in the setting of \eqref{eq:forw-bckw-BSVIE}--\eqref{eq:forward-SDE-intro}.

We now turn to the main contributions of the paper. Under suitable Malliavin regularity assumptions on the data $(\Psi,G)$, we prove a Hölder-type regularity estimate for $Z^\pi$ in the second time variable, averaged over the first one, of the form
\begin{gather}\label{eq:holder-estimate-discrete-intro}
    \mathbb E\bigg[
    \bigg(
    \sum_{k\colon t_k \leq s \wedge r} \Delta_k
    |Z^{\pi}(t_k,s)-Z^\pi(t_k,r)|^2
    \bigg)^{\frac p2}
    \bigg]
    \leq
    K\bigl(|s-r|^{\frac p2}+|\pi|^{\frac p2}\bigr),
    \quad
    s,r\in[0,T],
\end{gather}
where $K>0$ is independent of the partition, with $\Delta_k \coloneqq t_{k+1} - t_{k}$. The additional term $|\pi|^{p/2}$ is a purely discrete mesh-defect term and vanishes as the mesh size tends to zero. This estimate has two main consequences. First, by passing to the limit as $|\pi|\to0$, it yields an averaged Hölder regularity estimate for the martingale integrand of the BSVIE:
\begin{gather}\label{eq:holder-estimate-continuous-intro}
    \mathbb E\bigg[
    \bigg(
    \int_0^{s \wedge r}
    |Z(t,s)-Z(t,r)|^2\,dt
    \bigg)^{\frac p2}
    \bigg]
    \leq
    K|s-r|^{\frac p2},
    \qquad
    s,r\in[0,T].
\end{gather}
This can be viewed as the Volterra counterpart of the $L^p$-Hölder regularity estimate for the martingale integrand of a BSDE obtained in \cite{HuYaozhong2011MCFB}.

Second, we introduce a new explicit Euler scheme for the approximating BSDE system, inspired by the scheme of \cite{zhang_bsde_method}. Establishing a convergence rate outside the forward setting \eqref{eq:forw-bckw-BSVIE}--\eqref{eq:forward-SDE-intro} is more delicate, since regularity arguments based on the underlying forward process are no longer available. For BSDEs, the convergence analysis of Euler-type schemes relies on suitable time-regularity estimates for the martingale integrand $Z$; see \cite{zhang_bsde_method,dinunno2024deepoperatorbsde}. The same issue arises for the proposed BSVIE scheme. In this case, however, the required regularity must account for the two-time structure of $Z$ and therefore takes a different form. The case $p=2$ of the discrete estimate \eqref{eq:holder-estimate-discrete-intro} provides the regularity needed to prove a convergence rate for the scheme.

We emphasize that the results of this paper are not restricted to BSVIEs of the form \eqref{eq:forw-bckw-BSVIE}. Apart from the Malliavin regularity assumptions on the data, the free term $\Psi$ and the generator $G$ are not required to be induced by a forward stochastic differential equation or to have any specific functional form.

For simplicity, we restrict the presentation to scalar BSVIEs driven by a one-dimensional Brownian motion. The results can be easily extended to arbitrary dimensions.

The paper is organized as follows. Section \ref{section:preliminaries} introduces the notation and recalls the main results on BSDEs, BSVIEs, and the approximation of BSVIEs by systems of BSDEs. In Section \ref{section:holder-regularity-bsde-system}, we introduce the Malliavin calculus framework for the approximating BSDE system and prove the discrete regularity estimate \eqref{eq:holder-estimate-discrete-intro} under suitable Malliavin regularity assumptions on the data. Section \ref{section:holder-reg-bsvie} states the corresponding assumptions for the continuous BSVIE and proves \eqref{eq:holder-estimate-continuous-intro} by passing the discrete estimate to the limit as $|\pi|\to0$. We also provide examples of data satisfying these assumptions. Finally, Section \ref{section:numerical-approx} introduces an explicit Euler scheme for the BSDE system and derives a convergence rate under the previous assumptions. The Appendix contains auxiliary results and proofs omitted from the main text.

\section{Preliminaries}\label{section:preliminaries}

\subsection{Definitions and notation}

Let $B=(B(t))_{t\in[0,T]}$ be a one-dimensional Brownian motion defined on a complete probability space $(\Omega,\mathcal F,\mathbb P)$. We denote by $\mathbb F=(\mathcal F_t)_{t\in[0,T]}$ the filtration generated by $B$, augmented with the $\mathbb P$-null sets, and we assume that $\mathcal F=\mathcal F_T$. For $p\geq1$, we write $L^p(\Omega)$ for the space of all $\mathcal F_T$-measurable random variables $X:\Omega\to\mathbb R$ such that $\mathbb E|X|^p<\infty$. If $X\in L^1(\Omega)$ and $t\in[0,T]$, we use the shorthand notation $\mathbb E_t(X)\coloneqq \mathbb E\big[X\mid\mathcal F_t\big]$. Finally, we set
\[
    \Delta(0,T) \coloneqq \{(t,s)\in[0,T]^2: 0\leq t\leq s\leq T\}.
\]
We now introduce notation for one-parameter process spaces. For $q \geq 2$, we define:
\begin{itemize}
    \item $L^{q}\big(\Omega; L^2(0,T) \big)$ is the space of processes
    $\varphi\colon \Omega\times[0,T]\to \mathbb{R}$ that are
    $\mathcal F_{T}\otimes\mathcal B([0,T])$-measurable and satisfy
    \begin{gather*}
        \norm{\varphi}_{2,q}
        \coloneqq
        \mathbb E\!\left[
        \left(
        \int_{0}^{T} |\varphi(t)|^2\,dt
        \right)^{q/2}
        \right]^{1/q}
        <\infty.
    \end{gather*}
    Similarly, $L^{q}\big(\Omega; L^\infty(0,T) \big)$ is the space of processes
    $\varphi\colon \Omega\times[0,T]\to \mathbb{R}$ that are
    $\mathcal F_{T}\otimes\mathcal B([0,T])$-measurable and satisfy
    \begin{gather*}
        \norm{\varphi}_{\infty,q}
    \coloneqq
    \mathbb E\!\left[
    \operatorname*{ess\,sup}_{t\in[0,T]} |\varphi(t)|^q
    \right]^{1/q}
    <\infty.
    \end{gather*}
    
    \item $\mathcal{H}^{q}(0,T)$ is the subspace of $L^{q}\big(\Omega; L^2(0,T) \big)$ of $\mathbb{F}$-progressively measurable processes.
\end{itemize}

We also introduce notation for two-parameter processes
$\varphi \colon \Omega \times \Delta(0,T) \to \mathbb{R}$ such that, for a.e. $t \in [0,T]$, $[t,T]\ni s \mapsto \varphi(t,s)$ is $\mathbb{F}$-progressively measurable. For $q\geq2$, we define:
\begin{itemize}
    \item $\mathcal{H}_\Delta^{q}(0,T)$ is the space of such processes satisfying
    \begin{gather*}
        \norm{\varphi}_{\mathcal{H}_\Delta^{q}}^q
        \coloneqq
        \mathbb E\!\left[
        \left(
        \int_{0}^{T}\int_{t}^{T}|\varphi(t,s)|^2\,ds\,dt
        \right)^{q/2}
        \right]
        <\infty.
    \end{gather*}

    \item $\mathcal{T}_\Delta^{q}(0,T)$ is the space of such processes satisfying
    \begin{gather*}
        \norm{\varphi}_{\mathcal{T}_\Delta^{q}}^q
        \coloneqq
        \mathbb E\!\left[
        \left(
        \operatorname*{ess\,sup}_{t\in[0,T]}
\int_t^T|\varphi(t,s)|^2\,ds
        \right)^{q/2}
        \right]
        <\infty.
    \end{gather*}

    \item $\mathcal{S}_\Delta^{q}(0,T)$ is the space of such processes satisfying
    \begin{gather*}
        \norm{\varphi}_{\mathcal{S}_\Delta^{q}}^q
        \coloneqq
        \mathbb E\!\left[
        \left(
        \operatorname*{ess\,sup}_{s\in[0,T]}
\int_0^s|\varphi(t,s)|^2\,dt
        \right)^{q/2}
        \right]
        <\infty.
    \end{gather*}

    \item $\mathcal{U}_\Delta^{q}(0,T)$ is the space of such processes satisfying
    \begin{gather*}
        \norm{\varphi}_{\mathcal{U}_\Delta^{q}}^q
        \coloneqq
        \mathbb E\!\left[
        \left(
        \operatorname*{ess\,sup}_{(t,s)\in\Delta(0,T)}
|\varphi(t,s)|^2
        \right)^{q/2}
        \right]
        <\infty.
    \end{gather*}
\end{itemize}

We recall some basic notation from Malliavin calculus. We refer to \cite{NualartDavidTMCa} for a complete treatment.

Let $\mathbb H=L^2(0,T)$ be endowed with the usual norm $\norm{\cdot}_{\mathbb H}$. We denote by $C_p^\infty(\mathbb R^n)$ the class of infinitely continuously differentiable functions $g:\mathbb R^n\to\mathbb R$ such that $g$ and all its partial derivatives have polynomial growth.

Let $\mathcal S$ be the class of smooth cylindrical random variables of the form
\begin{gather}\label{eq:F}
    F
    =
    g\left(
    \int_0^T h_1(r)\,dB(r),
    \dots,
    \int_0^T h_n(r)\,dB(r)
    \right),
\end{gather}
where $g\in C_p^\infty(\mathbb R^n)$ and $h_1,\dots,h_n\in\mathbb H$. For $F$ as in \eqref{eq:F}, we define its Malliavin derivative $DF$ as the $\mathbb H$-valued random variable
\begin{flalign*}
    DF
    =
    \sum_{i=1}^{n}
    \partial_{x_i}g\left(
    \int_0^T h_1(r)\,dB(r),
    \dots,
    \int_0^T h_n(r)\,dB(r)
    \right)h_i.
\end{flalign*}
For $p\geq1$, we denote by $\mathbb D^{1,p}$ the closure of $\mathcal S$ with respect to the norm
\begin{gather*}
    \norm{F}_{\mathbb D^{1,p}}
    =
    \left(
    \mathbb E|F|^p
    +
    \mathbb E\norm{DF}_{\mathbb H}^p
    \right)^{\frac1p}.
\end{gather*}

Higher-order Malliavin derivatives are defined by iteration. If $F\in\mathcal S$, then $D^kF$ is an $\mathbb H^{\otimes k}$-valued random variable. For $k\geq1$ and $p\geq1$, we denote by $\mathbb D^{k,p}$ the completion of $\mathcal S$ with respect to the norm
\begin{gather*}
    \norm{F}_{\mathbb D^{k,p}}
    =
    \left(
    \mathbb E|F|^p
    +
    \sum_{j=1}^{k}
    \mathbb E\norm{D^jF}_{\mathbb H^{\otimes j}}^p
    \right)^{\frac1p}.
\end{gather*}
For stochastic processes, we write $u\in\mathbb L^{1,p}(0,T)$ if $u(t) \in \mathbb{D}^{1,p}$ and \begin{gather*} 
\mathbb E\left[ \left( \int_0^T |u(t)|^2\,dt \right)^{\frac p2} \right] + \mathbb E\left[ \left( \int_0^T \int_0^T |D_{\theta}u(t)|^2 \,dt\,d\theta \right)^{\frac p2} \right] <\infty. 
\end{gather*} 
We also write $\mathbb L_a^{1,p}(0,T)$ for the subspace of progressively measurable elements of $\mathbb L^{1,p}(0,T)$. For such processes, one may choose a progressively measurable version of the Malliavin derivative field $(D_\theta u_t)_{(\theta,t)\in[0,T]^2}$.

When an interval $(T',T)$, with $0\leq T'\leq T$, is indicated in the notation of any of the one-parameter process spaces introduced above, the corresponding definition is understood with the process-time variable restricted to $[T',T]$. For the Malliavin spaces, the Malliavin parameter still ranges over $[0,T]$.

Finally, when working with random fields $f\colon\Omega\times\mathbb R^m\to\mathbb R$, we use the notation $(D_\theta f)(X)$ to indicate that the Malliavin derivative is first taken with respect to $\omega$ for fixed $x$, and that the resulting random field is then evaluated at $x=X$, where $X$ is an $\mathbb R^m$-valued random variable. This should not be confused with the Malliavin derivative of the composition $f(X)$, denoted by $D_\theta f(X)$.

Throughout the paper, $C$ will be a general constant, which might change from line to line.

\subsection{BSVIEs and approximation by BSDE systems}

Throughout this subsection, fix $p\geq2$. Consider the BSVIE given in \eqref{eq:bsvie-intro}, where $G \colon\Omega\times\Delta(0,T)\times\mathbb R\times\mathbb R\to\mathbb R$ is
$\mathcal F\otimes\mathcal B(\Delta(0,T))\otimes\mathcal B(\mathbb R\times\mathbb R)$-measurable and such that, for every $(t,y,z)\in[0,T]\times\mathbb R\times\mathbb R$, the map
$[t,T]\ni s\mapsto G(t,s,y,z)$ is $\mathbb F$-progressively measurable. Let $\Psi\colon\Omega\times[0,T]\to\mathbb R$ be $\mathcal F\otimes\mathcal B([0,T])$-measurable.

The following result gives the well-posedness and $L^p$ a priori estimate used throughout the paper. Well-posedness follows from \cite[Theorem 3.7]{YongJiongmin2008Waro}; see also \cite[Lemma 3.2]{HamaguchiYushi2023AfaM}, while the estimate follows from \cite[Theorem 3.4]{HamaguchiYushi2023AfaM}. Although these results are stated there for Type-II BSVIEs, they apply to the present Type-I setting by taking the generator to be independent of $Z(s,t)$.

\begin{theorem}\label{thm:well-posedness-BSVIE}
    Suppose that $\Psi \in L^{p}\big(\Omega; L^2(0,T) \big)$, that
    $G^0 \coloneqq G(\cdot, \cdot, 0, 0) \in \mathcal{H}_\Delta^{p}(0,T)$, and that $G$ is uniformly Lipschitz in $(y,z)$. More precisely, assume that there exists a constant $[G]_L>0$ such that, $d\mathbb P\otimes dt\otimes ds$-a.e. on $\Omega\times\Delta(0,T)$,
    \begin{gather*}
        | G(t, s, y_1, z_1) - G(t, s, y_2, z_2)|
        \leq
        [G]_{L} \big( |y_1-y_2| + |z_1 - z_2|  \big),
    \end{gather*}
    for all $(y_1,z_1),(y_2,z_2)\in\mathbb R^2$. Then \eqref{eq:bsvie-intro} admits a unique solution $(Y,Z) \in \mathcal{H}^p(0,T) \times \mathcal{H}_{\Delta}^p(0,T)$.

    Moreover, define $Y(t,s)$ by
    \begin{flalign*}
        Y(t,s) \coloneqq \mathbb{E}_s \bigg[ \Psi(t) + \int_s^T G(t,r,Y(r),Z(t,r))dr \bigg], \quad  0 \leq t \leq s \leq T.
    \end{flalign*}
    Then there exists a constant $C>0$, depending only on $p$, $T$ and $[G]_L$, such that
    \begin{flalign}
       \mathbb{E}\Bigg[ \sup_{s \in [0,T]}  \bigg( \int_0^s |Y(t,s)|^2 dt \bigg)^{\frac{p}{2}} + \bigg( \int_0^T |Y(t)|^2 dt \bigg)^{\frac{p}{2}} + & \bigg( \int_0^T \int_t^T  |Z(t,s)|^2 ds dt \bigg)^{\frac{p}{2}} \Bigg] \leq C \Big( \norm{\Psi}_{2,p}^p + \norm{G^0}_{\mathcal{H}_\Delta^p}^p\Big).\label{eq:energy-estimates-BSVIE-solution}
    \end{flalign}
\end{theorem}

We next describe the approximation of the BSVIE by a system of BSDEs, following the approach of \cite{WANG20194926, HamaguchiYushi2023AfaM}. Let $\pi = \{0 = t_0 < \cdots < t_N = T\}$ be a partition of $[0,T]$, and write $\Delta_k \coloneqq t_{k+1}-t_k$ for $k=0,\dots,N-1$. We define the projection $\tau \colon [0,T]\to\pi$ by setting $\tau(r)=t_k$ for $r\in[t_k,t_{k+1})$, with the convention $\tau(T)=t_{N-1}$ at the terminal time.

Let $\big\{\Psi_k^\pi\big\}_{k=0}^{N-1}$ be a family of $\mathcal F_T$-measurable random variables, and let $\big\{G_k^\pi\big\}_{k=0}^{N-1}$ be a family of random fields $G_k^\pi \colon \Omega \times [t_k,T] \times \mathbb{R} \times \mathbb{R} \to \mathbb{R}$. More precisely, each $G_k^\pi$ is
$\mathcal F\otimes\mathcal B([t_k,T])\otimes\mathcal B(\mathbb R\times\mathbb R)$-measurable and, for every $(y,z)\in\mathbb R\times\mathbb R$, the map $[t_k,T]\ni s\mapsto G_k^\pi(s,y,z)$ is $\mathbb F$-progressively measurable. We assume that the generators vanish on the diagonal cell, namely
\begin{gather}\label{eq:vanish-diagonal}
    G_k^\pi(s,y,z) = 0 \qquad  s \in [t_k, t_{k+1}), \ (y,z) \in \mathbb{R}\times \mathbb{R}.
\end{gather}
We associate with these discrete data the piecewise constant objects
\begin{gather*}
    \Psi^\pi(t)
    \coloneqq
    \Psi_k^\pi, 
    \qquad
    G^\pi(t,s,y,z)
    \coloneqq
    G_k^\pi(s,y,z),
    \qquad t\in[t_k,t_{k+1}).
\end{gather*}
We also write $G^{\pi,0} \coloneqq \big( G^\pi(t, s, 0, 0) \big)_{(t,s) \in \Delta(0,T)}$. 

We now define the approximating BSDE system. For each fixed
$k\in\{0,\dots,N-1\}$, and for each $l\in\{k,\dots,N-1\}$, let
\begin{flalign}
Y^{\pi} (t_k, s) \label{def:BSDE-app-BSVIE}
&= Y^{\pi} (t_k, t_{l+1})
 + \int_s^{t_{l+1}} 
   G_k^\pi\Big(r, Y^{\pi} (t_l, r),
   Z^{\pi} (t_k, r)\Big)\, dr  - \int_s^{t_{l+1}} Z^{\pi} (t_k, r) \, dB(r),
\quad s \in [t_l, t_{l+1}], \\
Y^{\pi} (t_k, t_N) &= \Psi_k^\pi. \notag
\end{flalign}
Equivalently, $(Y^\pi(\tau(t), t), Z^\pi(\tau(t), s))_{(t,s) \in \Delta(0,T)}$ satisfies the BSVIE
\begin{flalign}\label{eq:bsvie-satisfied-by-discrete}
     Y^\pi(\tau(t), t)
    =
    \Psi^\pi(t)
    +
    \int_t^T 
    G^\pi\big(t,s,  Y^\pi(\tau(s), s), 
     Z^\pi(\tau(t),s)\big)\,ds 
    -
    \int_t^T  Z^\pi(\tau(t),s)\,dB(s).
\end{flalign}

The next result gives the well-posedness of the BSDE system \eqref{def:BSDE-app-BSVIE}, together with the corresponding a priori estimate. Its proof follows by adapting \cite[Theorem 3.4]{HamaguchiYushi2023AfaM} to the present Type-I setting.

\begin{theorem}\label{theorem:energy-estimates}
    Suppose that $\Psi^\pi \in L^{p}\big(\Omega; L^2(0,T) \big)$, that $G^{\pi,0} \in \mathcal{H}_{\Delta}^{p}(0,T)$, and that $G^\pi$ is uniformly Lipschitz in $(y,z)$, with Lipschitz constant $[G^\pi]_{L}$. Then the BSDE system \eqref{def:BSDE-app-BSVIE} admits a unique solution $(Y^\pi, Z^\pi )$ with
    \begin{gather*}
         \big(Y^\pi(t_k, \cdot), Z^\pi(t_k, \cdot) \big) \in \mathcal{H}^{p}(t_k,T) \times \mathcal{H}^p(t_k, T)
    \end{gather*}
    for each $k\in\{0,\dots,N-1\}$. Moreover, there exists a constant $C>0$, depending only on $p$, $T$ and $[G^\pi]_L$, such that 
    \begin{flalign*}
       \mathbb{E}\Bigg[ 
        \sup_{s\in[0,T]}
        \bigg(
        \sum_{k:t_k\leq s}
        \Delta_k &
        |Y^\pi(t_k,s)|^2
        \bigg)^{\frac p2} + 
        \bigg(  
        \int_{0}^{T}
        |Y^\pi(\tau(r),r)|^2\,dr
        \bigg)^{\frac p2}
         \\ &+
        \bigg(
        \sum_{k=0}^{N-1} \Delta_k
        \int_{t_{k+1}}^T
        |Z^\pi(t_k,s)|^2\,ds 
        \bigg)^{\frac p2}
        \Bigg]
        \leq
        C\bigg(
        \norm{\Psi^\pi}_{2,p}^p
        +
  \norm{G^{\pi,0}}_{\mathcal{H}_{\Delta}^p}^{p}
        \bigg).
    \end{flalign*}
\end{theorem}

As an immediate consequence, we obtain the following stability estimate for the BSDE system \eqref{def:BSDE-app-BSVIE}.

\begin{lemma}\label{lemma:stability-estimate-system-BSDE}
    Let $(\Psi^{\pi,i}, G^{\pi,i})$, $i=1,2$, be two families of discrete data satisfying the assumptions in Theorem \ref{theorem:energy-estimates}. Denote by $\big(Y_i^\pi, Z_i^\pi \big)$, $i=1,2$, the corresponding solutions of \eqref{def:BSDE-app-BSVIE}. Let us define
    \begin{flalign*}
        &\delta Y^\pi(\tau(t), s) \coloneqq (Y_1^\pi - Y_2^\pi)(\tau(t), s), \quad \delta Z^\pi(\tau(t), s) \coloneqq (Z_1^\pi - Z_2^\pi)(\tau(t), s) \\
        &\delta \Psi^\pi(\tau(t)) \coloneqq (\Psi^{\pi, 1} - \Psi^{\pi, 2})(\tau(t)), \quad \delta G^\pi(\tau(t), s, y, z) \coloneqq (G^{\pi, 1} - G^{\pi, 2})(\tau(t), s, y, z)
    \end{flalign*}
    Then there exists a constant $C>0$, depending only on $p$, $T$, $[G^{\pi,1}]_L$ and $[G^{\pi,2}]_L$, such that
\begin{align*}
    \mathbb{E}\Bigg[ 
    \sup_{s\in[0,T]}
    \bigg(
    \sum_{k:t_k\leq s} &
    \Delta_k
    |\delta Y^\pi(t_k,s)|^2
    \bigg)^{\frac p2}
    +
    \left(
    \int_0^T
    |\delta Y^\pi(\tau(r),r)|^2\,dr
    \right)^{\frac{p}{2}}
    +
    \left(
    \sum_{k=0}^{N-1}
    \Delta_k
    \int_{t_{k+1}}^T
    |\delta Z^\pi(t_k,r)|^2\,dr
    \right)^{\frac{p}{2}}
    \Bigg] \\ &
    \leq
    C\mathbb E\Bigg[
    \left(
    \sum_{k=0}^{N-1}
    \Delta_k
    |\delta\Psi^\pi(t_k)|^2
    \right)^{\frac{p}{2}}
    \Bigg]
     +
    C\mathbb E\Bigg[
    \left(
    \sum_{k=0}^{N-1}
    \Delta_k
    \int_{t_{k+1}}^T
    \Big|
    \delta G^\pi
    \big(
    t_k,r,
    Y_1^\pi(\tau(r),r),
    Z_1^\pi(t_k,r)
    \big)
    \Big|^2\,dr
    \right)^{\frac{p}{2}}
    \Bigg].
\end{align*}
\end{lemma}

The next theorem shows that the error between the BSVIE solution and the approximating BSDE system is controlled by the approximation errors in the free term and the generator. Indeed, the piecewise constant pair induced by the BSDE system satisfies \eqref{eq:bsvie-satisfied-by-discrete}, so the result follows from the $L^2$-stability estimate for BSVIEs; see \cite[Theorem 3.7]{YongJiongmin2008Waro}, specialized to the Type-I setting. Consequently, convergence of the discrete data implies convergence of the BSDE system to the BSVIE solution.

\begin{theorem}\label{thm:convergence-BSDE-system-BSVIE}
Suppose that $(\Psi, G)$ satisfy the assumptions in Theorem \ref{thm:well-posedness-BSVIE} with $p=2$, and let $(Y,Z)$ be the unique solution to \eqref{eq:bsvie-intro}. Let $(\Psi^\pi,G^\pi)$ satisfy the assumptions in Theorem \ref{theorem:energy-estimates} with $p=2$, with Lipschitz constant $[G^\pi]_L$ independent of $\pi$, and denote by $(Y^\pi,Z^\pi)$ the solution of the BSDE system \eqref{def:BSDE-app-BSVIE}.

Define the approximation errors
\begin{gather*}
    \varepsilon_{\Psi}^\pi
    \coloneqq
    \mathbb E\left[
    \int_0^T
    |\Psi(t)-\Psi^\pi(t)|^2\,dt
    \right], \quad \varepsilon_{G}^\pi
    \coloneqq
    \mathbb E\left[
    \int_0^T
    \left(
    \int_t^T
    \left|
    G(t,s,Y(s),Z(t,s))
    -
    G^\pi(t,s,Y(s),Z(t,s))
    \right|\,ds
    \right)^2
    dt \right].
\end{gather*}
Then there exists a constant $C>0$, independent of $\pi$, such that
\begin{align*}
    \mathbb E\Bigg[
    \int_0^T |Y(t)-Y^\pi(\tau(t),t)|^2\,dt
    +
    \int_0^T\int_t^T
    |Z(t,s)-Z^\pi(\tau(t),s)|^2\,ds\,dt 
    \Bigg]
    \leq
    C\big(\varepsilon_{\Psi}^\pi+\varepsilon_{G}^\pi\big). 
\end{align*}
In particular, if
$\varepsilon_{\Psi}^\pi+\varepsilon_{G}^\pi\to0$ as $|\pi|\to0$, then
\[
    (Y^\pi(\tau(\cdot),\cdot),Z^\pi(\tau(\cdot),\cdot))
    \longrightarrow
    (Y,Z) \quad \text{in $\mathcal H^2(0,T)\times\mathcal H_\Delta^2(0,T)$.}
\]
\end{theorem}

\section{Hölder regularity estimate for the BSDE system}\label{section:holder-regularity-bsde-system}

\subsection{Malliavin calculus for the BSDE system}

We now turn to the Hölder regularity of the martingale integrand of the approximating BSDE system. Using Malliavin calculus techniques inspired by the BSDE arguments in \cite{HuYaozhong2011MCFB}, we prove the estimate \eqref{eq:holder-estimate-discrete-intro}. We begin by stating the assumptions needed for this result.

\begin{assumption}\label{main assumption convergence rate BSVIEs, discrete setting}
    Fix $2 \leq p < \frac{q}{2}$, and let $\Pi=(\pi_n)_{n\geq1}$ be a sequence of partitions of $[0,T]$. For each $\pi\in\Pi$, let $(\Psi^{\pi},G^{\pi})$ denote the corresponding family of discrete data. We impose the following conditions.
    \begin{enumerate}[label=(\roman*)]
       \item For every $\pi \in \Pi$ and every $k\in\{0,\dots,N-1\}$, $\Psi_k^\pi \in \mathbb{D}^{2,2}$. Moreover, we have $\Psi^\pi \in L^{p}\big(\Omega; L^2(0,T) \big)$ and 
        \begin{flalign*}
       \mathfrak p_1^\Pi
        &\coloneqq
        \sup_{\pi\in\Pi}
        \sup_{\theta<\theta'}
        \frac{
        \|D_\theta\Psi^\pi-D_{\theta'}\Psi^\pi\|_{2,p}^p
        }{
        |\theta-\theta'|^{\frac p2}
        +
        |\pi|^{\frac p2}
        }
        <\infty, \qquad
        \mathfrak p_2^\Pi
        \coloneqq \sup_{\pi\in\Pi}
        \sup_{\theta\in[0,T]}
        \norm{D_\theta\Psi^\pi}_{\infty,q}
        <\infty, \\
        \mathfrak p_3^\Pi
        &\coloneqq \sup_{\pi\in\Pi}
        \sup_{\theta,\theta'\in[0,T]}
        \big\|D_{\theta'}D_\theta\Psi^\pi\big\|_{2,p}^p
        <\infty.
        \end{flalign*}

        \item The random fields $G_k^{\pi}$ have continuous first and second-order partial derivatives with respect to $y$ and $z$, which are uniformly bounded in all variables, $k$, and $\pi\in\Pi$. In addition $\sup_{\pi\in\Pi}\norm{G^{\pi,0}}_{\mathcal H_{\Delta}^q} < \infty$. We also assume that the diagonal vanishes, in the sense of \eqref{eq:vanish-diagonal}.

        \item For every $k\in\{0,\dots,N-1\}$
        and $(y,z)\in\mathbb R^2$, assume that $G_k^\pi(\cdot,y,z)$, $\partial_yG_k^\pi(\cdot,y,z)$ and $\partial_zG_k^\pi(\cdot,y,z)$ belong to $\mathbb L_a^{1,2}(t_k,T)$.
        
        Moreover, assume that, for every $\theta,\theta'\in[0,T]$, there exist non-negative random fields on $\Omega\times\Delta(0,T)$ such that, $\mathbb P$-a.s., for every
    $(t,s)\in\Delta(0,T)$ and $(y,z)\in\mathbb R^2$,
    \[
    |D_\theta G^\pi(t,s,y,z)|
    \leq
    \mathcal G_{\theta}^{1,\pi}(t,s),
    \qquad
    |D_\theta \partial_y G^\pi(t,s,y,z)|
    \leq
    \mathcal G_{\theta}^{y,\pi}(t,s),
    \]
    \[
    |D_\theta \partial_z G^\pi(t,s,y,z)|
    \leq
    \mathcal G_{\theta}^{z,\pi}(t,s),
    \qquad
    |D_{\theta'}D_\theta G^\pi(t,s,y,z)|
    \leq
    \mathcal G_{\theta,\theta'}^{2,\pi}(t,s).
    \]
    We assume that
    \begin{flalign}\label{eq:gamma_12_assumption}
        \gamma_1^\Pi
        &\coloneqq
        \sup_{\pi\in\Pi}
        \sup_{\theta\in[0,T]}
        \norm{\mathcal G_{\theta}^{1,\pi}}_{\mathcal T_\Delta^q}
        <\infty,  \quad
        \gamma_y^\Pi \coloneqq
        \sup_{\pi\in\Pi}
        \sup_{\theta\in[0,T]}
        \norm{\mathcal G_{\theta}^{y,\pi}}_{\mathcal S_\Delta^q}
        <\infty,
        \\
        \gamma_z^\Pi
        &\coloneqq
        \sup_{\pi\in\Pi}
        \sup_{\theta\in[0,T]}
        \norm{\mathcal G_{\theta}^{z,\pi}}_{\mathcal U_\Delta^q}
        <\infty,
        \quad
        \gamma_2^\Pi
        \coloneqq
        \sup_{\pi\in\Pi}
        \sup_{\theta,\theta'\in[0,T]}
        \norm{\mathcal G_{\theta,\theta'}^{2,\pi}}_{\mathcal H_\Delta^p}
        <\infty.
        \label{eq:gamma_34_assumption}
    \end{flalign}
        In addition, for every $\theta\in[0,T]$, $k$, and
        $(y,z)\in\mathbb R^2$, assume that
        $D_\theta G_k^\pi(\cdot,y,z)\in\mathbb L_a^{1,2}(t_k,T)$, and that
        $D_\theta G_k^\pi(s,\cdot,\cdot)$ has continuous partial derivatives with respect to
        $y$ and $z$, denoted by
        $\partial_yD_\theta G_k^\pi(s,y,z)$ and
        $\partial_zD_\theta G_k^\pi(s,y,z)$.
        
        Finally, assume that, for every $\theta<\theta'$, there exists a non-negative random field
        $\mathcal K_{\theta,\theta'}^\pi$ on $\Omega\times\Delta(0,T)$ such that, $\mathbb P$-a.s., for every
        $(t,s)\in\Delta(0,T)$ and $(y,z)\in\mathbb R^2$,
        \[
        |D_\theta G^\pi(t,s,y,z)-D_{\theta'}G^\pi(t,s,y,z)|
        \leq
        \mathcal K_{\theta,\theta'}^\pi(t,s).
        \]
        We assume that
        \[
        \kappa^\Pi
        \coloneqq
        \sup_{\pi\in\Pi}
        \sup_{\theta<\theta'}
        \frac{
        \mathbb E \bigg[
        \bigg(
        \sum_{k=0}^{N-1}
        \Delta_k
        \int_{t_{k+1}}^T
        \big|
        \mathcal K_{\theta,\theta'}^\pi(t_k,s)
        \big|^2\,ds
        \bigg)^{\frac p2}
        \bigg]
        }{
        |\theta-\theta'|^{\frac p2}
        +
        |\pi|^{\frac p2}
        }
        <\infty.
        \]
        Let us conclude by emphasizing that all dominating fields appearing above are understood to be piecewise constant on each interval $[t_k,t_{k+1})$ in their first time variable and to vanish on the same diagonal cell as $G^\pi$.
\end{enumerate}

\end{assumption}

\begin{remark}\label{remark:discrete-swapped-derivative-bounds}
The envelope bounds in \eqref{eq:gamma_12_assumption}--\eqref{eq:gamma_34_assumption} also control the partial derivatives of $D_\theta G^\pi$. Indeed, applying the fundamental theorem of calculus in the $y$-variable, taking the Malliavin derivative, and dividing by the increment gives
\[
    \frac{
    D_\theta G^\pi(t,s,y+h,z)-D_\theta G^\pi(t,s,y,z)
    }{h}
    =
    \frac{1}{h}
    \int_0^h
    D_\theta\partial_yG^\pi(t,s,y+u,z)\,du .
\]
The envelope bound on $D_\theta\partial_yG^\pi$ and the existence of $\partial_yD_\theta G^\pi$ then give
\[
    |\partial_yD_\theta G^\pi(t,s,y,z)|
    \leq
    \mathcal G_{\theta}^{y,\pi}(t,s),
\]
and the argument for the $z$-derivative is identical. 
\end{remark}

Throughout this subsection, we fix a partition $\pi\in\Pi$ and a family of discrete data $(\Psi^\pi,G^\pi)$ satisfying Assumption \ref{main assumption convergence rate BSVIEs, discrete setting}. All constants below are independent of $\pi$.

The next result gives the Malliavin differentiability of the solution to the approximating BSDE system and identifies $Z^\pi$ with the trace of the Malliavin derivative of $Y^\pi$. The proof follows by applying \cite[Proposition 5.3]{ElKarouiN.1997BSDE} recursively.

\begin{proposition}\label{prop:malliavin-differentiability-discrete-BSVIE}
    The BSDE system \eqref{def:BSDE-app-BSVIE} admits a unique solution $(Y^\pi, Z^\pi)$ such that, for each $k \in \{0, \dots, N-1\}$,
    $Y^\pi(t_k, \cdot)$ and $Z^\pi(t_k, \cdot)$ belong to $\mathbb{L}_a^{1,2}(t_k,T)$. Moreover, the Malliavin derivatives of the solution pair admit a version which satisfies the following linear BSDE system, of the same form as \eqref{def:BSDE-app-BSVIE}:
    \begin{flalign}
    D_\theta  Y^{\pi}(t_k, t) & =  D_\theta\Psi_k^\pi + \int_t^T \Big\{ \alpha_k^\pi(s)  D_\theta Y^{\pi}(\tau(s), s) + \beta_{k}^\pi(s) D_\theta Z^{\pi}(t_k, s)  \notag \\ & +  \big(D_\theta G_k^\pi \big) \Big(s, Y^{\pi}(\tau(s), s), Z^{\pi}(t_k, s)\Big)\Big\} ds - \int_t^T D_\theta Z^{\pi}(t_k, s) dB(s), \label{eq:DY-discrete-BSVIE} 
    \end{flalign}
    for $\theta \in [0,T]$ and $t \in [t_k, T] \cap [\theta, T]$, where
    \begin{flalign*}
            \alpha_{k}^\pi(s) &= \partial_y G_k^\pi\Big( s, Y^{\pi}(\tau(s), s), Z^{\pi}(t_k, s)\Big), \quad
            \beta_{k}^\pi(s) = \partial_{z} G_k^\pi\Big( s, Y^{\pi}(\tau(s), s), Z^{\pi}(t_k, s) \Big) .
    \end{flalign*}
    For $t \in [t_k, T] \cap [0, \theta)$, we have
    $D_\theta Y^{\pi}(t_k, t)=0$ and $D_\theta Z^{\pi}(t_k, t)=0$, the latter identity holding $dt\otimes d\mathbb P$-a.e. Finally, for each $k$, the diagonal Malliavin derivative of $Y^\pi(t_k,\cdot)$ provides a version of $Z^\pi(t_k,\cdot)$. More precisely,
    \begin{gather}\label{eq:malliavin-trace-bsde-system}
        Z^\pi(t_k, s) = D_s Y^\pi(t_k, s) \quad ds \otimes d\mathbb{P}\text{-a.e.}
    \end{gather}
\end{proposition}
The following estimate gives the corresponding energy bound for the Malliavin derivatives of the approximating BSDE system. It follows by applying the a priori estimate in Theorem \ref{theorem:energy-estimates} to the system satisfied by $(D_\theta Y^\pi,D_\theta Z^\pi)$, using the envelope bound on $D_\theta G^\pi$ from Assumption
\ref{main assumption convergence rate BSVIEs, discrete setting}.

\begin{proposition}\label{prop:malliavin-energy-estimates-discrete-BSVIE}
    For every $2\leq\lambda\leq q$ and every $\theta\in[0,T]$, there exists a constant $C>0$, independent of $\pi$ and $\theta$, such that
    \begin{flalign*}
        \mathbb{E} \Bigg[
        \bigg( 
        \int_{0}^{T} 
        |D_\theta Y^{\pi}(\tau(t),t)|^2\,dt 
        \bigg)^{\frac{\lambda}{2}}
        +
        \bigg(
        \sum_{k=0}^{N-1}
        \Delta_k
        \int_{t_{k+1}}^T 
        |D_\theta Z^{\pi}(t_k,s)|^2\,ds
        \bigg)^{\frac{\lambda}{2}}
        \Bigg] 
        \leq
        C \big(
        \norm{D_\theta \Psi^\pi}_{2,\lambda}^{\lambda}
        +
        \norm{\mathcal G_\theta^{1,\pi}}_{\mathcal{H}_{\Delta}^{\lambda}}^{\lambda}
        \big).
    \end{flalign*}
\end{proposition}

We are now ready to state the main result of this section, which gives a uniform $L^2$-Hölder regularity estimate for the martingale component of the approximating BSDE system. The key point is that the constant is independent of the partition, which is essential for passing to the limit as $|\pi|\to0$.

\begin{theorem}\label{thm:holder-regularity-system-BSDEs}
    There exists a constant $K>0$, independent of $\pi$,
    such that, for every $\pi\in\Pi$, one can choose a version of $Z^\pi$ satisfying
    \begin{gather*}
        \mathbb{E}\bigg[
        \bigg(
        \sum_{k\colon t_k \leq s \wedge r} \Delta_k
        | Z^{\pi}(t_k, s) - Z^\pi (t_k, r) |^2
        \bigg)^{\frac p2}
        \bigg]
        \leq
        K \big( |s-r|^{\frac p2} + |\pi|^{\frac p2} \big),
        \quad
       \forall s,r\in[0,T].
    \end{gather*}
\end{theorem}

The proof is given at the end of this section. Before turning to it, we first derive an explicit representation of $Z^\pi$ and collect the auxiliary estimates that will be used in the argument.

\subsection{An explicit representation of \(Z^\pi\) via Malliavin calculus}

We now derive a representation formula for $Z^\pi$. Using the Malliavin trace identity from the previous subsection, it is enough to represent $D_\theta Y^\pi$. The argument follows the same general strategy as in \cite{HuYaozhong2011MCFB}: we rewrite the Malliavin-differentiated equations as a linear multi-dimensional BSDE system and then apply a variation-of-constants formula.

Fix $\theta\in[0,T]$. For $k\in\{0,\dots,N-1\}$ and $l\in\{k+1,\dots,N-1\}$, define
\begin{flalign*}
    \alpha_{k,l}^\pi(u)
    &\coloneqq
    \partial_y G_k^\pi
    \Big(
    u,Y^\pi(t_l,u),Z^\pi(t_k,u)
    \Big)
    \mathbf 1_{[t_l,t_{l+1})}(u), \qquad
    \beta_k^\pi(u)
    \coloneqq
    \partial_zG_k^\pi
    \Big(
    u,Y^\pi(\tau(u),u),Z^\pi(t_k,u)
    \Big)
    \mathbf 1_{[t_{k+1},T]}(u).
\end{flalign*}
Then \eqref{eq:DY-discrete-BSVIE} can be written, for $t \in [t_k,T] \cap [\theta,T]$, as
\begin{flalign*}
    D_\theta Y^\pi(t_k,t)
    &=
    D_\theta\Psi_k^\pi
    +
    \int_t^T
    \bigg\{
    \sum_{l=k+1}^{N-1}
    \alpha_{k,l}^\pi(u)
    D_\theta Y^\pi(t_l,u)
    +
    \beta_k^\pi(u)
    D_\theta Z^\pi(t_k,u)
    \\
    &\qquad\quad
    +
    D_\theta G_k^\pi
    \Big(
    u,Y^\pi(\tau(u),u),Z^\pi(t_k,u)
    \Big)
    \bigg\}\,du
    -
    \int_t^T
    D_\theta Z^\pi(t_k,u)\,dB(u).
\end{flalign*}
It is convenient to view this family of equations in matrix notation. We extend $D_\theta Y^\pi(t_k,u)$ and $D_\theta Z^\pi(t_k,u)$ by zero for $u<t_k$, and set
\begin{flalign*}
    D_\theta Y^{\pi} (u) &= \begin{pmatrix}
        D_\theta Y^{\pi}(t_0, u) & \cdots & D_\theta Y^{\pi}(t_{N-1}, u)
    \end{pmatrix}^T \in \mathbb{R}^{N}, \\
    D_\theta Z^{\pi} (u) &= \begin{pmatrix}
        D_\theta Z^{\pi}(t_0, u) & \cdots & D_\theta Z^{\pi}(t_{N-1}, u)
    \end{pmatrix}^T \in \mathbb{R}^{N}, \\
    D_\theta \Psi^\pi &= \begin{pmatrix}
        D_\theta \Psi_0^\pi & \cdots & D_\theta \Psi_{N-1}^\pi
    \end{pmatrix}^T \in \mathbb{R}^{N}.
\end{flalign*}
The coefficient matrices are given by
\begin{gather*}
  \mathcal A^\pi(u)_{k,j}
    =
    \begin{cases}
        \alpha_{k,j}^\pi(u), & j\geq k+1, \\
        0, & j\leq k,
    \end{cases}
    \qquad
    \mathcal B^\pi(u)_{k,j}
    =
    \begin{cases}
        \beta_k^\pi(u), & j=k, \\
        0, & j\neq k,
    \end{cases}
\end{gather*}
for $k,j\in\{0,\dots,N-1\}$. The source term $\mathcal C_\theta^\pi(u) = \big(\mathcal C_\theta^\pi(u)_k\big)_{0\leq k\leq N-1}$ is given by
\begin{gather*}
    \mathcal C_\theta^\pi(u)_k
    =
    \big(D_\theta G_k^\pi\big)
    \Big(u,Y^\pi(\tau(u),u),Z^\pi(t_k,u)\Big)
    \mathbf 1_{[t_{k+1},T]}(u).
\end{gather*}
With this notation, the differentiated equations are represented by the following vector relation, whose $k$th component is understood only for $t\in[t_k,T]\cap[\theta,T]$:
\begin{flalign*}
    D_\theta Y^{\pi} (t)
    =
    D_\theta \Psi^\pi
    +
    \int_t^T
    \big\{
    \mathcal A^\pi(u) D_\theta Y^{\pi} (u)
    +
    \mathcal B^\pi(u) D_\theta Z^{\pi} (u)
    +
    \mathcal C_{\theta}^\pi(u)
    \big\}\,du
    -
    \int_t^T D_\theta Z^{\pi} (u)\,dB(u).
\end{flalign*}
We now introduce the fundamental matrix associated with this linear system. Let $\Phi^\pi$ be the $N\times N$-dimensional solution of the linear SDE system
\begin{gather*}
    d \Phi^\pi(t)
    =
    \Phi^\pi(t)
    \big[
    \mathcal A^\pi(t)\,dt
    +
    \mathcal B^\pi(t)\,dB(t)
    \big],
    \qquad
    \Phi^\pi(0)=I_N.
\end{gather*}
Set
\[
    U^\pi(t,u)
    \coloneqq
    \Phi^\pi(t)^{-1}\Phi^\pi(u),
    \qquad
    0\leq t\leq u\leq T.
\]
The triangular structure and the supports of the coefficients imply that $U^\pi(t,u)_{k,j}=0$ for $j>k$ and $u\leq t_j$. Hence, the zero extensions do not contribute to the componentwise variation-of-constants argument, and, for $t\in[t_k,T]\cap[\theta,T]$, we obtain
\begin{flalign*}
    D_\theta Y^\pi(t_k,t)
    =
    \mathbb E_t\Bigg[
        \sum_{j=k}^{N-1}
        U^\pi(t,T)_{k,j}
        D_\theta\Psi_j^\pi
        +
        \int_t^T
        \sum_{j=k}^{N-1}
        U^\pi(t,u)_{k,j}
        \mathcal C_\theta^\pi(u)_j\,du
    \Bigg].
\end{flalign*}
In the sequel, using \eqref{eq:malliavin-trace-bsde-system}, we take the version of $Z^\pi$ obtained by setting $\theta=t$ in the preceding representation formula.

\begin{remark}
Since $\mathcal A^\pi$ is strictly upper triangular and $\mathcal B^\pi$ is diagonal, the fundamental matrix $U^\pi(t,u)$ is upper triangular. Hence,
$U^\pi(t,u)_{k,j}=0$ for $j<k$.

For $j\geq k$, its entries satisfy
\begin{flalign*}
    d U^\pi(t,u)_{k,j}
    &=
    \Big(
    \sum_{l=k}^{j-1}
    \alpha_{l,j}^{\pi}(u)
    U^\pi(t,u)_{k,l}
    \Big)\,du
    +
    \beta_j^{\pi}(u)
    U^\pi(t,u)_{k,j}\,dB(u), \qquad U^\pi(t,t)_{k,j}
    =
    \delta_{kj}.
\end{flalign*}
In particular,
\begin{gather}\label{eq:U-diagonal-expression}
    U^\pi(t,u)_{k,k}
    =
    \exp\Big(
    \int_t^u \beta_k^{\pi}(v)\,dB(v)
    -
    \frac{1}{2}\int_t^u \beta_k^{\pi}(v)^2\,dv
    \Big).
\end{gather}
For $j\geq k+1$, variation of constants and the support of
$\alpha_{l,j}^{\pi}$ yield
\begin{gather}\label{eq:equation-U-offdg}
    U^\pi(t,u)_{k,j}
    =
    \int_{[t,u]\cap[t_j,t_{j+1})}
    \mathcal E_j^\pi(v,u)
    \sum_{l=k}^{j-1}
    \alpha_{l,j}^{\pi}(v)
    U^\pi(t,v)_{k,l}\,dv,
\end{gather}
where $\mathcal E_j^\pi(v,u)\coloneqq U^\pi(v,u)_{j,j}$.
\end{remark}

\subsection{Estimates for the Fundamental Matrix}\label{subsec:estimates-fundamental-matrix}

We now state the key estimates for the fundamental matrix $U^\pi$ that will be used in the proof of Theorem \ref{thm:holder-regularity-system-BSDEs}. The proofs are given in \ref{section:appendix-proofs-fundamental-matrix}.

The first estimate reflects the triangular structure of the system: diagonal entries of $U^\pi$ are of order one, whereas off-diagonal entries carry one factor of the corresponding mesh size.

\begin{lemma}\label{lemma:bound_U_delta}
Let $\lambda>1$. There exists a constant $C>0$, independent of $\pi$ and $t$, such that, for all $0\leq k\leq j\leq N-1$,
\begin{gather}\label{eq:U-entry-bound}
     \mathbb{E}\Big[\sup_{t \leq u \leq T} |U^\pi(t,u)_{k,j}|^{\lambda} \ \big| \ \mathcal{F}_{t}\Big]^{\frac{1}{\lambda}}
    \leq
    C\,\mathbf{1}_{\{k=j\}}
    + C\,\Delta_j\,\mathbf{1}_{\{k<j\}}, \quad \text{$\mathbb{P}$-a.s.}
\end{gather} 
\end{lemma}

We also need bounds on the Malliavin derivatives of the linear coefficients appearing in the differentiated system.

\begin{lemma}\label{lemma:energy-Dalpha-Dbeta}
Let $u \in [0,T]$. Then, for every $2 \leq \lambda \leq q$, there exists a constant $C>0$, independent of $u$ and $\pi$, such that
\begin{flalign*}
    \mathbb E\bigg[
    \bigg(
    \sum_{k=0}^{N-1}
    \Delta_k
    \int_{t_{k+1}}^T  |D_u\alpha_k^\pi(s)|^2\,ds
    \bigg)^{\frac{\lambda}{2}}
    \bigg]^{\frac1\lambda} 
    +&
    \mathbb E\bigg[
    \bigg(
    \sum_{k=0}^{N-1}
    \Delta_k
    \int_{t_{k+1}}^T |D_u\beta_k^\pi(s)|^2\,ds
    \bigg)^{\frac{\lambda}{2}}
    \bigg]^{\frac1\lambda}
    \leq
    C \Big\{
    \mathfrak p_2^\Pi
    +
    \gamma_1^\Pi
    +
    \gamma_y^\Pi
    +
    \gamma_z^\Pi
    \Big\}.
\end{flalign*}
\end{lemma}

The following two estimates control the Malliavin derivative of the fundamental matrix. We separate the diagonal and off-diagonal entries, since the latter require the additional mesh-weighted structure.

\begin{lemma}\label{lemma:diagonal-DU-bound}
Let $t\in[0,T]$ and $u\in[0,t)$. Let $1\leq \lambda<2$ and $2<\eta<q$. Then there exists a constant $C>0$, independent of $\pi$, $t$ and $u$, such that
\begin{flalign*}
    \mathbb{E}\bigg[
    \bigg(
    \sum_{k=0}^{N-1}
    \Delta_k
    \mathbb{E}_{u}\Big[
    \sup_{v \in [t, T]}
    \big|D_u U^\pi(t,v)_{k,k}\big|^{\lambda}
    \Big]^{\frac{2}{\lambda}}
    \bigg)^{\frac{\eta}{2}}
    \bigg]^{\frac{1}{\eta}}
    \leq
    C \Big\{
    \mathfrak p_2^\Pi
    +
    \gamma_1^\Pi
    +
    \gamma_z^\Pi
    \Big\}.
\end{flalign*}
\end{lemma}

\begin{lemma}\label{lemma:bound_D_theta_off_diagonal_rho1}
Let $t\in[0,T]$ and $u\in[0,t)$. Let $1\leq \lambda<2$ and $2< \eta<q$. Then there exists a constant $C>0$, independent of $\pi$, $t$ and $u$, such that
\begin{flalign}
\mathbb{E}\bigg[\bigg(\sum_{k = 0}^{N-1} \Delta_k
\sum_{j=k+1}^{N-1} 
\frac{1}{\Delta_j} 
\mathbb{E}_u \big[
\sup_{v \in [t, T]}
|D_u U^\pi(t,v)_{k,j}|^\lambda
\big]^{\frac{2}{\lambda}}
\bigg)^{\frac{\eta}{2}} \bigg]^{\frac{1}{\eta}}
 \leq
C \Big\{
\mathfrak p_2^\Pi
+
\gamma_1^\Pi
+
\gamma_y^\Pi
+
\gamma_z^\Pi
\Big\}.
\label{ineq:bound_D_theta_off_diagonal_rho1}
\end{flalign}
\end{lemma}

The final estimate in this subsection concerns the time regularity of the propagator. This will be used to control the increments of the representation formula in the time variable.

\begin{lemma}\label{lemma:time-regularity-U}
    Let $0\leq r\leq s\leq v\leq T$, and define $\delta_j(r,s)\coloneqq |[r,s]\cap[t_j,t_{j+1})|$, where $|\cdot|$ denotes Lebesgue measure. Then there exists a constant $C>0$, independent of $\pi$, $r$, $s$ and $v$, such that, for every $0 \leq k \leq N-1$,
    \begin{flalign}\label{eq:time-regularity-diagonal-U}
      \mathbb{E}_r \Big[ \big| U^\pi(s,v)_{k,k} - U^\pi(r,v)_{k,k} \big|^2 \Big]
      \leq C |s-r|.  
    \end{flalign}
    Moreover, for every $0\leq k<j\leq N-1$,
    \begin{flalign}\label{eq:time-regularity-off-diagonal-U}
      \mathbb{E}_r \Big[ \big| U^\pi(s,v)_{k,j} - U^\pi(r,v)_{k,j} \big|^2 \Big]
      \leq C\Big( \Delta_j^2 |s-r| + \delta_j(r,s)^2 \Big).
    \end{flalign}
\end{lemma}

\subsection{Proof of Theorem \ref{thm:holder-regularity-system-BSDEs}}

Assume $s\geq r$, without loss of generality. By
\eqref{eq:malliavin-trace-bsde-system}, for $r\geq t_k$,
\begin{flalign}\notag
Z^\pi(t_k,s)-Z^\pi(t_k,r)
&=
D_sY^\pi(t_k,s)-D_rY^\pi(t_k,r)
\\
&=
\big[D_sY^\pi(t_k,s)-D_rY^\pi(t_k,s)\big]
+
\big[D_rY^\pi(t_k,s)-D_rY^\pi(t_k,r)\big].
\label{eq:decomp-Z-BSDE-system}
\end{flalign}

We first estimate the first term. Applying Lemma
\ref{lemma:stability-estimate-system-BSDE} to the BSDE systems satisfied by
$D_sY^\pi$ and $D_rY^\pi$, and then Assumption
\ref{main assumption convergence rate BSVIEs, discrete setting} with
$\theta=r$ and $\theta'=s$, gives
\begin{flalign*}
\mathbb E\bigg[
\bigg(
\sum_{k\colon t_k\leq r}
\Delta_k
\big|D_sY^\pi(t_k,s)-D_rY^\pi(t_k,s)\big|^2
\bigg)^{\frac p2}
\bigg]
&\leq
C\mathbb E\bigg[
\bigg(
\sum_{k=0}^{N-1}
\Delta_k
\big|D_s\Psi_k^\pi-D_r\Psi_k^\pi\big|^2
\bigg)^{\frac p2}
+
\bigg(
\sum_{k=0}^{N-1}
\Delta_k
\int_{t_{k+1}}^T
\big|\mathcal C_s^\pi(v)_k-\mathcal C_r^\pi(v)_k\big|^2\,dv
\bigg)^{\frac p2}
\bigg]
\\
&\quad\leq
C
\big(
\mathfrak p_1^\Pi+\kappa^\Pi
\big)
\big(
|s-r|^{\frac p2}
+
|\pi|^{\frac p2}
\big).
\end{flalign*}
It remains to estimate the second term in \eqref{eq:decomp-Z-BSDE-system}. Define
\[
\Lambda_{u,v}(k)
\coloneqq
\mathbb{E}_u\Bigg[
    \sum_{j=k}^{N-1} U^\pi(v,T)_{k,j}D_r\Psi_j^\pi \Bigg],  \qquad \Xi_{u,v,w}(k) 
\coloneqq
\mathbb{E}_u\Bigg[
    \int_v^T
    \sum_{j=k}^{N-1} U^\pi(w,a)_{k,j} \, \mathcal C_r^\pi(a)_j\,da
\Bigg]
\]
for $k \in \{0, \dots, N-1\}$. We then have
\begin{flalign*}
D_rY^\pi(t_k,s)-D_rY^\pi(t_k,r)
&=
\Lambda_{s,s}(k)+\Xi_{s,s,s}(k)
-
\Lambda_{r,r}(k)-\Xi_{r,r,r}(k)
\\
&=
\big[\Lambda_{s,s}(k)-\Lambda_{r,s}(k)\big]
+
\big[\Lambda_{r,s}(k)-\Lambda_{r,r}(k)\big]
\\
&\quad
+
\big[\Xi_{s,s,s}(k)-\Xi_{r,s,s}(k)\big]
+
\big[\Xi_{r,s,s}(k)-\Xi_{r,s,r}(k)\big]
+
\big[\Xi_{r,s,r}(k)-\Xi_{r,r,r}(k)\big]
\\
&\eqqcolon I_1(k)+I_2(k)+I_3(k)+I_4(k)+I_5(k).
\end{flalign*}
We show that
\begin{flalign}\label{eq:I_i_ineq}
    \mathbb E\bigg[
    \bigg(
    \sum_{k \colon t_k \leq r}
    \Delta_k |I_i(k)|^2
    \bigg)^{\frac p2}
    \bigg]
    \leq
    K |s-r|^{\frac p2},
    \quad \forall i \in \{1, \dots, 5\},
\end{flalign}
with $K$ independent of $s$, $r$, and $\pi$.

\medskip

\textit{Estimate for $I_1$.} By definition,
\begin{flalign*}
    I_1(k)
    &=
    \mathbb E_s\left[
    \sum_{j=k}^{N-1}
    U^\pi(s,T)_{k,j}D_r\Psi_j^\pi
    \right]
    -
    \mathbb E_r\left[
    \sum_{j=k}^{N-1}
    U^\pi(s,T)_{k,j}D_r\Psi_j^\pi
    \right],
\end{flalign*}
so $I_1$ is a martingale increment over $(r,s)$. The Clark--Ocone formula and the Burkholder--Davis--Gundy inequality in the weighted Hilbert space $\mathbb R^N$, endowed with the norm
$\|x\|_\Delta^2\coloneqq\sum_{k=0}^{N-1}\Delta_k|x_k|^2$, yield
\begin{flalign*}
    \mathbb E\bigg[
    \bigg(
    \sum_{k\colon t_k\leq r}
    \Delta_k |I_1(k)|^2
    \bigg)^{\frac p2}
    \bigg]
    \leq
    C|s-r|^{\frac p2}
    \sup_{u\in[r,s]}
    \mathbb E\bigg[
    \bigg(
    \sum_{k\colon t_k\leq r}
    \Delta_k |J_1(k)|^2
    \bigg)^{\frac p2}
    \bigg],
\end{flalign*}
where
\begin{gather*}
    J_1(k)
    \coloneqq
    \mathbb E_u\Bigg[
    D_u\Bigg(
    \sum_{j=k}^{N-1}
    U^\pi(s,T)_{k,j}D_r\Psi_j^\pi
    \Bigg)
    \Bigg].
\end{gather*}
It therefore suffices to prove
\begin{flalign}\label{eq:bound-on-J1}
    \mathbb E\bigg[
    \bigg(
    \sum_{k\colon t_k\leq r}
    \Delta_k |J_1(k)|^2
    \bigg)^{\frac p2}
    \bigg]
    \leq C,
\end{flalign}
uniformly in $u$, $r$, $s$, and $\pi$. By the Malliavin product rule,
\begin{flalign*}
    J_1(k)
    &=
    \mathbb E_u\Bigg[
    \sum_{j=k}^{N-1}
    D_uU^\pi(s,T)_{k,j}D_r\Psi_j^\pi
    \Bigg]
    +
    \mathbb E_u\Bigg[
    \sum_{j=k}^{N-1}
    U^\pi(s,T)_{k,j}D_uD_r\Psi_j^\pi
    \Bigg]
    \\
    &\eqqcolon
    J_2(k)+J_3(k).
\end{flalign*}
Set
\begin{gather}\label{eq:holder-exponents-main-proof}
    \lambda\coloneqq\frac{q}{q-1},
    \qquad
    \eta\coloneqq\frac{pq}{q-p}.
\end{gather}
Since $2\leq p<q/2$, we have $1\leq\lambda<2$ and $2<\eta<q$.

Using conditional Hölder's inequality with exponents $\lambda$ and $q$, weighted Cauchy--Schwarz for the off-diagonal sum, Hölder's inequality in expectation, and Lemmas
\ref{lemma:diagonal-DU-bound} and
\ref{lemma:bound_D_theta_off_diagonal_rho1}, we obtain
\begin{flalign*}
    \mathbb E\bigg[
    \bigg(
    \sum_{k\colon t_k\leq r}
    \Delta_k |J_2(k)|^2
    \bigg)^{\frac p2}
    \bigg]
    \leq
    C
    \big(\mathfrak p_2^\Pi\big)^p
    \big(
    \mathfrak p_2^\Pi
    +
    \gamma_1^\Pi
    +
    \gamma_y^\Pi
    +
    \gamma_z^\Pi
    \big)^p.
\end{flalign*}
For $J_3$, conditional Cauchy--Schwarz, Lemma
\ref{lemma:bound_U_delta}, weighted Cauchy--Schwarz in the off-diagonal sum, and conditional Jensen's inequality give
\begin{flalign*}
    \mathbb E\bigg[
    \bigg(
    \sum_{k\colon t_k\leq r}
    \Delta_k |J_3(k)|^2
    \bigg)^{\frac p2}
    \bigg]
    \leq
    C
    \mathbb E\bigg[
    \bigg(
    \sum_{j=0}^{N-1}
    \Delta_j
    |D_uD_r\Psi_j^\pi|^2
    \bigg)^{\frac p2}
    \bigg]
    \leq
    C\mathfrak p_3^\Pi.
\end{flalign*}
Combining the estimates for $J_2$ and $J_3$ proves
\eqref{eq:bound-on-J1}, and hence \eqref{eq:I_i_ineq} for $i=1$.

\medskip

\textit{Estimate for $I_2$.} We split the contribution into its diagonal and off-diagonal parts:
\begin{flalign*}
    I_2(k)
    &=
    \mathbb{E}_r\Big[
    \big(U^\pi(s,T)_{k,k}-U^\pi(r,T)_{k,k}\big)
    D_r\Psi_k^\pi
    \Big]
    +
    \sum_{j=k+1}^{N-1}
    \mathbb{E}_r\Big[
    \big(U^\pi(s,T)_{k,j}-U^\pi(r,T)_{k,j}\big)
    D_r\Psi_j^\pi
    \Big]
    \\
    &\eqqcolon
    I_2^{\mathrm{dg}}(k)
    +
    I_2^{\mathrm{off-dg}}(k).
\end{flalign*}
By conditional Cauchy--Schwarz and Lemma
\ref{lemma:time-regularity-U},
\begin{gather*}
    \big|I_2^{\mathrm{dg}}(k)\big|^2
    \leq
    C|s-r|
    \mathbb E_r\big[
    |D_r\Psi_k^\pi|^2
    \big],
\end{gather*}
whereas
\begin{flalign*}
    \big|I_2^{\mathrm{off-dg}}(k)\big|^2
    &\leq
    C\bigg(
    \sum_{j=k+1}^{N-1}
    \Big(
    \Delta_j|s-r|^{\frac12}
    +
    \delta_j(r,s)
    \Big)
    \mathbb E_r\big[
    |D_r\Psi_j^\pi|^2
    \big]^{\frac12}
    \bigg)^2
    \leq
    C|s-r|
    \sum_{j=k+1}^{N-1}
    \Delta_j
    \mathbb E_r\big[
    |D_r\Psi_j^\pi|^2
    \big].
\end{flalign*}
Here, the last inequality follows from Cauchy--Schwarz, together with
$\sum_j\delta_j(r,s)\leq |s-r|$ and
$\delta_j(r,s)\leq\Delta_j$.

Consequently, summing over $k$, using
$\sum_{k\colon t_k\leq r}\Delta_k\leq T$, conditional Jensen's inequality,
and Assumption
\ref{main assumption convergence rate BSVIEs, discrete setting}, we obtain
\begin{flalign*}
    \mathbb E\bigg[
    \bigg(
    \sum_{k\colon t_k\leq r}
    \Delta_k|I_2(k)|^2
    \bigg)^{\frac p2}
    \bigg]
    &\leq
    C|s-r|^{\frac p2}
    \mathbb E\bigg[
    \bigg(
    \sum_{j=0}^{N-1}
    \Delta_j
    |D_r\Psi_j^\pi|^2
    \bigg)^{\frac p2}
    \bigg]
    \leq
    C|s-r|^{\frac p2}
    \big(\mathfrak p_2^\Pi\big)^p.
\end{flalign*}
Thus, \eqref{eq:I_i_ineq} holds for $i=2$.

\medskip

\textit{Estimate for $I_3$.} We first observe that
\begin{flalign*}
    I_3(k)
    &=
    \mathbb E_s\left[
    \int_s^T
    \sum_{j=k}^{N-1}
    U^\pi(s,v)_{k,j}\mathcal C_r^\pi(v)_j\,dv
    \right]
    -
    \mathbb E_r\left[
    \int_s^T
    \sum_{j=k}^{N-1}
    U^\pi(s,v)_{k,j}\mathcal C_r^\pi(v)_j\,dv
    \right].
\end{flalign*}
As in the estimate for $I_1$, the Clark--Ocone formula and the
Burkholder--Davis--Gundy inequality in
$(\mathbb R^N,\|\cdot\|_\Delta)$ reduce the proof of
\eqref{eq:I_i_ineq} for $i=3$ to showing that
\begin{flalign}\label{eq:bound-on-J4}
    \mathbb E\bigg[
    \bigg(
    \sum_{k\colon t_k\leq r}
    \Delta_k |J_4(k)|^2
    \bigg)^{\frac p2}
    \bigg]
    \leq C, \qquad \text{where} \qquad J_4(k)
    \coloneqq
    \mathbb E_u\Bigg[
    D_u\Bigg(
    \int_s^T
    \sum_{j=k}^{N-1}
    U^\pi(s,v)_{k,j}\mathcal C_r^\pi(v)_j\,dv
    \Bigg)
    \Bigg],
\end{flalign}
uniformly in $u\in[r,s)$, $r$, $s$, and $\pi$. By the Malliavin product
rule,
\begin{flalign*}
    J_4(k)
    &=
    \mathbb E_u\Bigg[
    \int_s^T
    \sum_{j=k}^{N-1}
    D_uU^\pi(s,v)_{k,j}\mathcal C_r^\pi(v)_j\,dv
    \Bigg]
    +
    \mathbb E_u\Bigg[
    \int_s^T
    \sum_{j=k}^{N-1}
    U^\pi(s,v)_{k,j}D_u\mathcal C_r^\pi(v)_j\,dv
    \Bigg]
    \\
    &\eqqcolon
    J_5(k)+J_6(k).
\end{flalign*}
Using conditional Hölder's inequality with the exponents $\lambda$ and $q$
from \eqref{eq:holder-exponents-main-proof}, weighted Cauchy--Schwarz for
the off-diagonal terms, Hölder's inequality in expectation, the envelope
bound for $\mathcal C_r^\pi$, and Lemmas
\ref{lemma:diagonal-DU-bound} and
\ref{lemma:bound_D_theta_off_diagonal_rho1}, we obtain
\begin{flalign*}
    \mathbb E\bigg[
    \bigg(
    \sum_{k\colon t_k\leq r}
    \Delta_k |J_5(k)|^2
    \bigg)^{\frac p2}
    \bigg]
    \leq
    C
    \big(\gamma_1^\Pi\big)^p
    \big(
    \mathfrak p_2^\Pi
    +
    \gamma_1^\Pi
    +
    \gamma_y^\Pi
    +
    \gamma_z^\Pi
    \big)^p.
\end{flalign*}

For $J_6$, conditional Cauchy--Schwarz, Lemma
\ref{lemma:bound_U_delta}, weighted Cauchy--Schwarz for the off-diagonal
terms, and conditional Jensen's inequality give
\begin{flalign*}
    \mathbb E\bigg[
    \bigg(
    \sum_{k\colon t_k\leq r}
    \Delta_k |J_6(k)|^2
    \bigg)^{\frac p2}
    \bigg]
    \leq
    C
    \mathbb E\bigg[
    \bigg(
    \sum_{j=0}^{N-1}
    \Delta_j
    \int_{t_{j+1}}^T
    |D_u\mathcal C_r^\pi(v)_j|^2\,dv
    \bigg)^{\frac p2}
    \bigg].
\end{flalign*}
By the Malliavin chain rule, for $v\in[t_{j+1},T]$,
\begin{flalign*}
    D_u\mathcal C_r^\pi(v)_j
    &=
    \big(D_uD_rG_j^\pi\big)
    \Big(v,Y^\pi(\tau(v),v),Z^\pi(t_j,v)\Big)
    \\
    &\quad+
    \big(\partial_yD_rG_j^\pi\big)
    \Big(v,Y^\pi(\tau(v),v),Z^\pi(t_j,v)\Big)
    D_uY^\pi(\tau(v),v)
    \\
    &\quad+
    \big(\partial_zD_rG_j^\pi\big)
    \Big(v,Y^\pi(\tau(v),v),Z^\pi(t_j,v)\Big)
    D_uZ^\pi(t_j,v).
\end{flalign*}
Consequently, the envelope bounds, Hölder's inequality, and Proposition
\ref{prop:malliavin-energy-estimates-discrete-BSVIE} yield
\begin{flalign*}
    &\mathbb E\bigg[
    \bigg(
    \sum_{k\colon t_k\leq r}
    \Delta_k |J_6(k)|^2
    \bigg)^{\frac p2}
    \bigg]
    \leq
    C\Bigg\{
    \big(\gamma_2^\Pi\big)^p
    +
    \bigg(
    \big(\gamma_y^\Pi\big)^p
    +
    \big(\gamma_z^\Pi\big)^p
    \bigg)
    \bigg(
    \big(\mathfrak p_2^\Pi\big)^p
    +
    \big(\gamma_1^\Pi\big)^p
    \bigg)
    \Bigg\}.
\end{flalign*}
Combining the estimates for $J_5$ and $J_6$ proves
\eqref{eq:bound-on-J4}, and hence \eqref{eq:I_i_ineq} for $i=3$.

\medskip

\textit{Estimate for $I_4$.} We now estimate $I_4(k).$ We write
\begin{flalign*}
    I_4(k)
    &=
    \int_s^T \mathbb{E}_r\Big[\big( U^\pi(s,v)_{k,k} - U^\pi(r,v)_{k,k} \big) \mathcal C_r^\pi(v)_k\,\Big]dv 
    +  \sum_{j=k+1}^{N-1} \int_s^T \mathbb{E}_r\Big[\big( U^\pi(s,v)_{k,j} - U^\pi(r,v)_{k,j} \big) \mathcal C_r^\pi(v)_j\,\Big]dv  \\
    &\eqqcolon I_4^{\text{dg}}(k) + I_4^{\text{off-dg}}(k).
\end{flalign*}
Using conditional Cauchy--Schwarz, Lemma \ref{lemma:time-regularity-U}, and Cauchy--Schwarz in the time variable,
\begin{flalign*}
    |I_4^{\mathrm{dg}}(k)|^2
    &\leq
    \bigg(
    \int_s^T
    \mathbb{E}_r\Big[\big| U^\pi(s,v)_{k,k} - U^\pi(r,v)_{k,k} \big|^2 \Big]^{\frac{1}{2}}
    \mathbb{E}_r\Big[\big|\mathcal C_r^\pi(v)_k\big|^2\Big]^{\frac{1}{2}}
    dv
    \bigg)^2  \\
    &\leq
    C|s-r|
    \int_s^T
    \mathbb E_r\Big[
    |\mathcal C_r^\pi(v)_k|^2
    \Big]
    dv.
\end{flalign*}
Hence, by conditional Jensen's inequality and the tower property,
\begin{flalign*}
    \mathbb E\bigg[
    \bigg(
    \sum_{k \colon t_k \leq r}
    \Delta_k
    |I_4^{\mathrm{dg}}(k)|^2
    \bigg)^{\frac p2}
    \bigg]
    &\leq
    C|s-r|^{\frac p2}
    \mathbb E\bigg[
    \bigg(
    \sum_{k \colon t_k \leq r}
    \Delta_k
    \int_s^T
    \mathbb E_r\Big[
    |\mathcal C_r^\pi(v)_k|^2
    \Big]
    dv
    \bigg)^{\frac p2}
    \bigg] \\
    &\leq
    C|s-r|^{\frac p2}
    \mathbb E\bigg[
    \bigg(
    \sum_{k=0}^{N-1}
    \Delta_k
    \int_{t_{k+1}}^T
    |\mathcal C_r^\pi(v)_k|^2
    dv
    \bigg)^{\frac p2}
    \bigg] \\
    &\leq
    C|s-r|^{\frac p2}
    \big(\gamma_1^\Pi\big)^p .
\end{flalign*}
For the off-diagonal terms, conditional Cauchy--Schwarz and Lemma \ref{lemma:time-regularity-U} give
\begin{flalign*}
    |I_4^{\mathrm{off-dg}}(k)|^2
    &\leq
    \bigg(
    \sum_{j=k+1}^{N-1}
    \int_s^T
    \mathbb{E}_r\Big[
    \big| U^\pi(s,v)_{k,j} - U^\pi(r,v)_{k,j} \big|^2
    \Big]^{\frac{1}{2}}
    \mathbb{E}_r\Big[
    |\mathcal C_r^\pi(v)_j|^2
    \Big]^{\frac{1}{2}}
    \,dv
    \bigg)^2 \\
    &\leq
    C
    \bigg(
    \sum_{j=k+1}^{N-1}
    \Big(
    \Delta_j |s-r|^{\frac12}
    +
    \delta_j(r,s)
    \Big)
    \int_s^T
    \mathbb{E}_r\Big[
    |\mathcal C_r^\pi(v)_j|^2
    \Big]^{\frac{1}{2}}
    \,dv
    \bigg)^2 .
\end{flalign*}
We estimate the two contributions separately. By Cauchy--Schwarz in the sum over $j$ and in the time variable,
\begin{flalign*}
    \bigg(
    \sum_{j=k+1}^{N-1}
    \Delta_j |s-r|^{\frac12}
    \int_s^T
    \mathbb{E}_r\Big[
    |\mathcal C_r^\pi(v)_j|^2
    \Big]^{\frac{1}{2}}
    \,dv
    \bigg)^2
    &\leq
    C|s-r|
    \sum_{j=k+1}^{N-1}
    \Delta_j
    \int_s^T
    \mathbb{E}_r\Big[
    |\mathcal C_r^\pi(v)_j|^2
    \Big]\,dv .
\end{flalign*}
Moreover, using $\sum_j\delta_j(r,s)\leq |s-r|$, $\delta_j(r,s)\leq \Delta_j$, and again Cauchy--Schwarz in time,
\begin{flalign*}
    \bigg(
    \sum_{j=k+1}^{N-1}
    \delta_j(r,s)
    \int_s^T 
    \mathbb{E}_r\Big[
    |\mathcal C_r^\pi(v)_j|^2 
    \Big]^{\frac{1}{2}}
    \,dv
    \bigg)^2
    &\leq
    C
    \bigg(
    \sum_{j=k+1}^{N-1}
    \delta_j(r,s)
    \bigg)
    \bigg(
    \sum_{j=k+1}^{N-1}
    \delta_j(r,s)
    \int_s^T
    \mathbb{E}_r\Big[
    |\mathcal C_r^\pi(v)_j|^2
    \Big]\,dv
    \bigg) \\
    &\leq
    C|s-r|
    \sum_{j=k+1}^{N-1}
    \Delta_j
    \int_s^T
    \mathbb{E}_r\Big[
    |\mathcal C_r^\pi(v)_j|^2
    \Big]\,dv .
\end{flalign*}
Therefore,
\begin{flalign*}
    |I_4^{\mathrm{off-dg}}(k)|^2
    &\leq
    C|s-r|
    \sum_{j=k+1}^{N-1}
    \Delta_j
    \int_s^T
    \mathbb E_r\Big[
    |\mathcal C_r^\pi(v)_j|^2
    \Big]\,dv .
\end{flalign*}
Hence, by conditional Jensen's inequality and the tower property,
\begin{flalign*}
    \mathbb E\bigg[
    \bigg(
    \sum_{k \colon t_k \leq r}
    \Delta_k
    |I_4^{\mathrm{off-dg}}(k)|^2
    \bigg)^{\frac p2}
    \bigg]
    &\leq
    C|s-r|^{\frac p2}
    \mathbb E\bigg[
    \bigg(
    \sum_{k \colon t_k \leq r}
    \Delta_k
    \sum_{j=k+1}^{N-1}
    \Delta_j
    \int_s^T
    \mathbb E_r\Big[
    |\mathcal C_r^\pi(v)_j|^2
    \Big]\,dv
    \bigg)^{\frac p2}
    \bigg] \\
    &\leq
    C|s-r|^{\frac p2}
    \mathbb E\bigg[
    \bigg(
    \sum_{j=0}^{N-1}
    \Delta_j
    \int_{t_{j+1}}^T
    \mathbb E_r\Big[
    |\mathcal C_r^\pi(v)_j|^2
    \Big]\,dv
    \bigg)^{\frac p2}
    \bigg] \\
    &\leq
    C|s-r|^{\frac p2}
    \mathbb E\bigg[
    \bigg(
    \sum_{j=0}^{N-1}
    \Delta_j
    \int_{t_{j+1}}^T
    |\mathcal C_r^\pi(v)_j|^2
    dv
    \bigg)^{\frac p2}
    \bigg] \\
    &\leq
    C|s-r|^{\frac p2}
    \big(\gamma_1^\Pi\big)^p .
\end{flalign*}
Combining the diagonal and off-diagonal estimates gives \eqref{eq:I_i_ineq} for $i=4$.

\medskip

\textit{Estimate for $I_5$.} We separate the diagonal and off-diagonal contributions:
\begin{flalign*}
    I_5(k)
    &=
    -
    \int_r^s
    \mathbb E_r\Big[
    U^\pi(r,v)_{k,k}\mathcal C_r^\pi(v)_k
    \Big]\,dv
    -
    \int_r^s
    \sum_{j=k+1}^{N-1}
    \mathbb E_r\Big[
    U^\pi(r,v)_{k,j}\mathcal C_r^\pi(v)_j
    \Big]\,dv
    \\
    &\eqqcolon
    I_5^{\mathrm{dg}}(k)
    +
    I_5^{\mathrm{off-dg}}(k).
\end{flalign*}
By conditional Cauchy--Schwarz, Lemma
\ref{lemma:bound_U_delta}, and Cauchy--Schwarz in time and in the sum over
$j$,
\begin{flalign*}
    \big|I_5^{\mathrm{dg}}(k)\big|^2
    &\leq
    C|s-r|
    \int_r^s
    \mathbb E_r\big[
    |\mathcal C_r^\pi(v)_k|^2
    \big]\,dv,
    \\
    \big|I_5^{\mathrm{off-dg}}(k)\big|^2
    &\leq
    C|s-r|
    \sum_{j=k+1}^{N-1}
    \Delta_j
    \int_r^s
    \mathbb E_r\big[
    |\mathcal C_r^\pi(v)_j|^2
    \big]\,dv.
\end{flalign*}
Consequently, summing over $k$, using conditional Jensen's inequality, the
tower property, and the envelope bound for $\mathcal C_r^\pi$, we obtain
\begin{flalign*}
    \mathbb E\bigg[
    \bigg(
    \sum_{k\colon t_k\leq r}
    \Delta_k|I_5(k)|^2
    \bigg)^{\frac p2}
    \bigg]
    &\leq
    C|s-r|^{\frac p2}
    \mathbb E\bigg[
    \bigg(
    \sum_{j=0}^{N-1}
    \Delta_j
    \int_{t_{j+1}}^T
    |\mathcal C_r^\pi(v)_j|^2\,dv
    \bigg)^{\frac p2}
    \bigg]
    \\
    &
    \leq
    C|s-r|^{\frac p2}
    \big(\gamma_1^\Pi\big)^p.
\end{flalign*}
Thus, \eqref{eq:I_i_ineq} holds for $i=5$.

We have therefore proved \eqref{eq:I_i_ineq} for every $i\in\{1,\dots,5\}$. Combining these estimates with the decomposition of
$D_rY^\pi(t_k,s)-D_rY^\pi(t_k,r)$, and then with \eqref{eq:decomp-Z-BSDE-system}, yields the desired estimate. The constants involved depend only on the constants in Assumption \ref{main assumption convergence rate BSVIEs, discrete setting}, and are therefore independent of $\pi$, $s$ and $r$. This completes the proof.

\section{Hölder regularity estimate for the solution of the BSVIE}\label{section:holder-reg-bsvie}

We now prove the Hölder regularity estimate for the $Z$ component of the BSVIE. The proof combines the uniform estimates obtained for the approximating BSDE systems with a passage to the limit. This is done through cell-averaged approximations of the data, which allow us to recover the BSVIE without imposing additional pointwise regularity in the first time variable. Later, in Subsection \ref{section:examples}, we provide two examples of data satisfying the assumptions used below.

\subsection{Assumptions and main result}

We first introduce the assumptions needed for the Hölder regularity estimate. They are stated directly at the level of the BSVIE data and are chosen so that the corresponding cell-averaged BSDE systems satisfy the uniform discrete assumptions from the previous section.

\begin{assumption}\label{assumption-Holder-BSVIE}
    Fix $2\leq p<\frac q2$. We impose the following conditions.

    \begin{enumerate}[label=(\roman*)]
    \item $\Psi(t) \in \mathbb{D}^{2,2}$ for all $t \in [0,T]$. Moreover, we have $\Psi \in L^p(\Omega; L^2(0,T))$ and
\begin{flalign*}
    \mathfrak p_1
    &\coloneqq 
    \sup_{\theta < \theta'}
    \frac{
    \|D_\theta\Psi-D_{\theta'}\Psi\|_{2,p}^{p}
    }{
    |\theta-\theta'|^{\frac p2}
    }
    < \infty, \qquad
    \mathfrak p_2
    \coloneqq
    \sup_{\theta \in [0,T]}
    \norm{D_\theta\Psi}_{\infty,q}
    < \infty, \\
    \mathfrak p_3
    &\coloneqq 
    \sup_{\theta, \theta' \in [0,T]}
    \big\|D_{\theta'}D_\theta\Psi\big\|_{2,p}^{p}
    < \infty.
\end{flalign*}

   \item The generator $G$ has continuous and uniformly bounded first and second-order partial derivatives with respect to $y$ and $z$, uniformly in $(t,s)$. Moreover, $\norm{G^{0}}_{\mathcal H_\Delta^q}<\infty$.

    \item Under conditions (i) and (ii), let $(Y,Z)$ denote the solution to the BSVIE \eqref{eq:bsvie-intro}. For almost every $t\in[0,T]$ and $(y,z)\in\mathbb R^2$, assume that $G(t,\cdot,y,z)$, $\partial_yG(t,\cdot,y,z)$ and $\partial_zG(t,\cdot,y,z)$ belong to $\mathbb L_a^{1,2}(t,T)$.

    Moreover, assume that, for every $\theta,\theta'\in[0,T]$, there exist non-negative random fields on $\Omega\times\Delta(0,T)$ such that, $\mathbb P$-a.s., for every $(t,s)\in\Delta(0,T)$ and $(y,z)\in\mathbb R^2$,
\[
    |D_\theta G(t,s,y,z)|
    \leq
    \mathcal G_\theta^1(t,s),
    \qquad
    |D_\theta \partial_y G(t,s,y,z)|
    \leq
    \mathcal G_\theta^y(t,s),
\]
\[
    |D_\theta \partial_z G(t,s,y,z)|
    \leq
    \mathcal G_\theta^z(t,s),
    \qquad
    |D_{\theta'}D_\theta G(t,s,y,z)|
    \leq
    \mathcal G_{\theta,\theta'}^2(t,s).
\]
We assume that
\begin{flalign*}
    \gamma_1
    &\coloneqq
    \sup_{\theta\in[0,T]}
    \norm{\mathcal G_\theta^1}_{\mathcal T_\Delta^q}
    <\infty,
    \quad
    \gamma_y
    \coloneqq
    \sup_{\theta\in[0,T]}
    \norm{\mathcal G_\theta^y}_{\mathcal S_\Delta^q}
    <\infty,
    \\
    \gamma_z
    &\coloneqq
    \sup_{\theta\in[0,T]}
    \norm{\mathcal G_\theta^z}_{\mathcal U_\Delta^q}
    <\infty,
    \quad
    \gamma_2
    \coloneqq
    \sup_{\theta,\theta'\in[0,T]}
    \norm{\mathcal G_{\theta,\theta'}^2}_{\mathcal H_\Delta^p}
    <\infty.
\end{flalign*}
    In addition, for every $\theta,t\in[0,T]$ and $(y,z)\in\mathbb R^2$, assume that $D_\theta G(t,\cdot,y,z)\in\mathbb L_a^{1,2}(t,T)$. Moreover, for every $s\in[t,T]$, assume that $D_\theta G(t,s,\cdot,\cdot)$ has continuous partial derivatives with respect to $y$ and $z$, denoted by $\partial_yD_\theta G(t,s,y,z)$ and $\partial_zD_\theta G(t,s,y,z)$.

    Finally, assume that, for every $\theta<\theta'$, there exists a non-negative random field $\mathcal K_{\theta,\theta'}$ on $\Omega\times\Delta(0,T)$ such that, $\mathbb P$-a.s., for every $(t,s)\in\Delta(0,T)$ and every $(y,z)\in\mathbb R^2$,
    \[
        |D_\theta G(t,s,y,z)-D_{\theta'}G(t,s,y,z)|
        \leq
        \mathcal K_{\theta,\theta'}(t,s).
    \]
    We assume that
    \[
        \kappa
        \coloneqq
        \sup_{\theta<\theta'}
        \frac{
        \mathbb E
        \left[
        \left(
        \int_0^T
        \int_t^T
        |\mathcal K_{\theta,\theta'}(t,s)|^2\,ds\,dt
        \right)^{\frac p2}
        \right]
        }{
        |\theta-\theta'|^{\frac p2}
        }
        <\infty.
    \]
    \end{enumerate}
\end{assumption}

\begin{remark}\label{remark:swapped-derivative-bounds}
As in the discrete setting, the pointwise envelope bounds on $D_\theta\partial_yG$ and $D_\theta\partial_zG$ also control the partial derivatives of $D_\theta G$. More precisely, for every $(t,s)\in\Delta(0,T)$, $\theta\in[0,T]$ and $(y,z)\in\mathbb R^2$, we have
\begin{gather*}
    |\partial_yD_\theta G(t,s,y,z)|
    \leq
    \mathcal G_\theta^y(t,s),
    \qquad
    |\partial_zD_\theta G(t,s,y,z)|
    \leq
    \mathcal G_\theta^z(t,s),
    \qquad \mathbb P\text{-a.s.}
\end{gather*}
\end{remark}

Under these assumptions, the $Z$ component of the BSVIE solution satisfies the following Hölder-type estimate in the second time variable, after averaging in $L^2$ over the first time variable.

\begin{theorem}\label{thm:holder-regularity-BSVIE}
    Let $(\Psi, G)$ satisfy Assumption \ref{assumption-Holder-BSVIE}, and let $(Y, Z)$ be the unique solution to the BSVIE \eqref{eq:bsvie-intro}. Then there exists a version of $Z$ such that
    \begin{gather*}
        \mathbb E\bigg[
        \bigg(
        \int_0^{s\wedge r}
        |Z(t,s)-Z(t,r)|^2\,dt
        \bigg)^{\frac p2}
        \bigg]
        \leq K |s-r|^{\frac p2},
        \qquad s,r\in[0,T],
    \end{gather*}
    where $K>0$ is independent of $s$ and $r$.
\end{theorem}

The usual dyadic proof of the Kolmogorov continuity theorem (see \cite[Theorem 2.1]{revuz}) can be adapted to the estimate in Theorem \ref{thm:holder-regularity-BSVIE}, yielding the following pathwise consequence.

\begin{corollary}\label{cor:holder-continuity-BSVIE}
    Suppose that Assumption \ref{assumption-Holder-BSVIE} holds for some $p>2$,
    and let $0<\gamma<\frac12-\frac1p$. Then there exists a version of $Z$,
    still denoted by $Z$, and a random variable $C_\gamma$ satisfying
    $\mathbb E[|C_\gamma|^p]<\infty$, such that, almost surely,
    \begin{gather*}
        \bigg(
        \int_0^{s\wedge r}
        |Z(t,s)-Z(t,r)|^2\,dt
        \bigg)^{\frac12}
        \leq
        C_\gamma |s-r|^\gamma,
        \qquad s,r\in[0,T].
    \end{gather*}
\end{corollary}

The proofs of Theorem \ref{thm:holder-regularity-BSVIE} and Corollary \ref{cor:holder-continuity-BSVIE} are given later in this section. We first introduce the cell-averaged BSDE systems used to approximate the BSVIE and to pass the discrete regularity estimate to the limit.

\subsection{Approximation by cell-averaged BSDE systems}

To prove Theorem \ref{thm:holder-regularity-BSVIE}, we use the approximating BSDE systems introduced earlier, with data obtained by averaging the BSVIE coefficients over the first time variable. As in the discrete setting, the generator is set equal to zero on the diagonal cell. The following lemma shows that the corresponding BSDE systems converge to the BSVIE solution in the relevant $L^p(\Omega;L^2)$ norm.

\begin{lemma}\label{lemma:convergence-cell-average-data}
    Suppose that Assumption \ref{assumption-Holder-BSVIE} holds, and let $(Y,Z)$ be the unique solution to the BSVIE \eqref{eq:bsvie-intro}. 
    
    Let $\pi$ be a partition of $[0,T]$. For each $k\in\{0,\dots,N-1\}$, define
    \begin{gather}\label{eq:terminal-data-averaged}
        \Psi_k^\pi
        \coloneqq
        \frac{1}{\Delta_k}
        \int_{t_k}^{t_{k+1}}\Psi(t)\,dt.
    \end{gather}
    Moreover, for $0\leq k\leq N-1$ and $s\in[t_k,T]$, define
    \begin{gather}\label{eq:generator-data-averaged}
        G_k^\pi(s,y,z)
        \coloneqq
        \begin{cases}
        0, & s\in[t_k,t_{k+1}), \\[0.4em]
        \displaystyle
        \frac{1}{\Delta_k}
        \int_{t_k}^{t_{k+1}}
        G(r,s,y,z)\,dr,
        & s\in[t_{k+1},T].
        \end{cases}
    \end{gather}
    Let $\big(Y^\pi(t_k,\cdot),Z^\pi(t_k,\cdot)\big)_{k=0}^{N-1}$ be the solution of the BSDE system associated with the data $(\Psi^\pi,G^\pi)$. Then
    \begin{flalign}\label{eq:lp-conv-BSDE-BSVIE}
        \mathbb E\bigg[
        \int_0^T |Y(t)-Y^\pi(\tau(t),t)|^2\,dt
        \bigg]
        +
        \mathbb E\bigg[
        \int_0^T\int_t^T
        |Z(t,s)-Z^\pi(\tau(t),s)|^2\,ds\,dt
        \bigg]
        \longrightarrow 0
    \end{flalign}
    whenever $|\pi|\to0$.
\end{lemma}

\begin{proof}
By Assumption \ref{assumption-Holder-BSVIE} and Jensen's inequality, the cell-averaged data $(\Psi^\pi,G^\pi)$ satisfy the assumptions of Theorem \ref{theorem:energy-estimates}. Hence the corresponding BSDE system admits a unique solution. Moreover, Theorem \ref{thm:convergence-BSDE-system-BSVIE} reduces the proof of \eqref{eq:lp-conv-BSDE-BSVIE} to showing that
\begin{gather*}
    \varepsilon_\Psi^\pi,\varepsilon_G^\pi
    \longrightarrow0,
    \qquad |\pi|\to0.
\end{gather*}
The convergence of $\varepsilon_\Psi^\pi$ follows from the strong convergence of cell-average approximations in $L^2(\Omega\times(0,T))$.

For the generator error, set
\begin{gather*}
    R^\pi(t,s)
    \coloneqq
    G(t,s,Y(s),Z(t,s))
    -
    G^\pi(\tau(t),s,Y(s),Z(t,s)).
\end{gather*}
By the definition of $\varepsilon_G^\pi$ and Cauchy--Schwarz,
\begin{flalign*}
    \varepsilon_G^\pi
    &=
    \mathbb E\int_0^T
    \bigg(
    \int_t^T |R^\pi(t,s)|\,ds
    \bigg)^2dt
    \leq
    T\mathbb E\int_0^T\int_t^T
    |R^\pi(t,s)|^2\,ds\,dt.
\end{flalign*}
It is therefore enough to prove
\begin{gather}\label{eq:L2-R-G}
    \mathbb E\int_0^T\int_t^T
    |R^\pi(t,s)|^2\,ds\,dt
    \longrightarrow0.
\end{gather}

For almost every $(t,s)\in\Delta(0,T)$, we have $t<s$. Let $k$ be such that $t\in[t_k,t_{k+1})$. For $|\pi|$ sufficiently small, $s\geq t_{k+1}$, and hence, for fixed $(y,z)\in\mathbb R^2$,
\begin{gather*}
    G^\pi(\tau(t),s,y,z)
    =
    \frac{1}{\Delta_k}
    \int_{t_k}^{t_{k+1}}
    G(r,s,y,z)\,dr
    \longrightarrow
    G(t,s,y,z)
\end{gather*}
for a.e. $(t,s,\omega)\in\Delta(0,T)\times\Omega$, by the Lebesgue differentiation theorem. Applying this first on a countable dense subset of $\mathbb R^2$ and then using the uniform Lipschitz property in $(y,z)$ shows that the convergence holds simultaneously for all $(y,z)$ outside a common null set. Therefore,
\begin{gather}\label{eq:pointwise-convergence}
    R^\pi(t,s)\longrightarrow0
    \qquad
    \text{for a.e. }(t,s,\omega)\in\Delta(0,T)\times\Omega.
\end{gather}

It remains to obtain an integrable dominating term. For $s\in[0,T]$, define
\begin{gather*}
    \overline G_s(t)
    \coloneqq
    G(t,s,0,0)\mathbf 1_{[0,s]}(t),
    \qquad t\in[0,T],
\end{gather*}
and let $\mathcal M_1$ denote the Hardy--Littlewood maximal operator in the first time variable. By the definition of the cell average,
\begin{gather*}
    |G^\pi(\tau(t),s,0,0)|
    \leq
    \mathcal M_1\overline G_s(t).
\end{gather*}
Consequently, the Lipschitz property gives
\begin{flalign*}
    |R^\pi(t,s)|^2
    \leq
    C\bigg(
    |G(t,s,0,0)|^2
    +
    |\mathcal M_1\overline G_s(t)|^2
    +
    |Y(s)|^2
    +
    |Z(t,s)|^2
    \bigg).
\end{flalign*}
Moreover, by Fubini's theorem and the Hardy--Littlewood maximal inequality, applied pathwise in the first time variable,
\begin{flalign*}
    \int_0^T\int_t^T
    |\mathcal M_1\overline G_s(t)|^2\,ds\,dt
    &=
    \int_0^T\int_0^s
    |\mathcal M_1\overline G_s(t)|^2\,dt\,ds
    \\
    &\leq
    C\int_0^T\int_0^T
    |\overline G_s(t)|^2\,dt\,ds
    \\
    &=
    C\int_0^T\int_0^s
    |G(t,s,0,0)|^2\,dt\,ds.
\end{flalign*}
Thus, Assumption \ref{assumption-Holder-BSVIE} and
\eqref{eq:energy-estimates-BSVIE-solution} imply
\begin{flalign*}
    \mathbb E\int_0^T\int_t^T
    \big(
    |G(t,s,0,0)|^2
    +
    |\mathcal M_1\overline G_s(t)|^2
    +
    |Y(s)|^2
    +
    |Z(t,s)|^2
    \big)\,ds\,dt
    <\infty.
\end{flalign*}
Combining this domination with \eqref{eq:pointwise-convergence} and applying the dominated convergence theorem on $\Omega\times\Delta(0,T)$ yields \eqref{eq:L2-R-G}. Hence $\varepsilon_G^\pi\to0$, which completes the proof.
\end{proof}

\subsection{Proof of Theorem \ref{thm:holder-regularity-BSVIE}}

Let $\Pi=(\pi_n)_{n\geq1}$ be a sequence of partitions of $[0,T]$ such that $|\pi_n|\to0$. For each $\pi\in\Pi$, let $(\Psi_k^\pi,G_k^\pi)_{0\leq k\leq N-1}$ be the cell-averaged data defined in \eqref{eq:terminal-data-averaged}--\eqref{eq:generator-data-averaged}.

\medskip
\noindent\textit{Step 1: Assumption \ref{assumption-Holder-BSVIE} implies Assumption \ref{main assumption convergence rate BSVIEs, discrete setting}.}

We show that the cell-averaged data satisfy Assumption
\ref{main assumption convergence rate BSVIEs, discrete setting} uniformly
along $\Pi$.

For the terminal data, the linearity and closedness of the Malliavin derivative give
\begin{gather*}
    D_\theta\Psi_k^\pi
    =
    \frac{1}{\Delta_k}
    \int_{t_k}^{t_{k+1}}D_\theta\Psi(t)\,dt,
    \qquad
    D_{\theta'}D_\theta\Psi_k^\pi
    =
    \frac{1}{\Delta_k}
    \int_{t_k}^{t_{k+1}}D_{\theta'}D_\theta\Psi(t)\,dt.
\end{gather*}
Consequently, Jensen's inequality yields
\begin{gather*}
    \|\Psi^\pi\|_{2,p}
    \leq
    \|\Psi\|_{2,p},
    \qquad
    \mathfrak p_1^\Pi\leq\mathfrak p_1,
    \qquad
    \mathfrak p_2^\Pi\leq\mathfrak p_2,
    \qquad
    \mathfrak p_3^\Pi\leq\mathfrak p_3.
\end{gather*}
Thus, condition \emph{(i)} holds uniformly along $\Pi$.

For the generator, differentiation with respect to $(y,z)$ commutes with
the cell average. For instance, for $s\in[t_{k+1},T]$,
\begin{gather*}
    \partial_yG_k^\pi(s,y,z)
    =
    \frac{1}{\Delta_k}
    \int_{t_k}^{t_{k+1}}
    \partial_yG(a,s,y,z)\,da,
\end{gather*}
and the same identity holds for $\partial_zG_k^\pi$ and all second-order
partial derivatives, while these derivatives vanish on
$[t_k,t_{k+1})$. Hence the first and second-order partial derivatives of
$G^\pi$ are uniformly bounded, the generators are uniformly Lipschitz in
$(y,z)$, and the diagonal vanishes. Moreover, Jensen's inequality gives
\begin{gather*}
    \sup_{\pi\in\Pi}
    \norm{G^{\pi,0}}_{\mathcal H_\Delta^q}
    \leq
    \norm{G^0}_{\mathcal H_\Delta^q}.
\end{gather*}
This verifies condition \emph{(ii)}.

Similarly, for $s\in[t_{k+1},T]$,
\begin{gather*}
    \big(D_\theta G_k^\pi\big)(s,y,z)
    =
    \frac{1}{\Delta_k}
    \int_{t_k}^{t_{k+1}}
    D_\theta G(a,s,y,z)\,da,
\end{gather*}
while $D_\theta G_k^\pi=0$ on $[t_k,t_{k+1})$. The same averaging
identities hold for $D_\theta\partial_yG_k^\pi$,
$D_\theta\partial_zG_k^\pi$, and $D_{\theta'}D_\theta G_k^\pi$.
The required membership in the spaces $\mathbb L_a^{1,2}$, as well as
the continuity of the partial derivatives of $D_\theta G_k^\pi$ with
respect to $(y,z)$, therefore follow from the corresponding properties
of $G$ and Jensen's inequality.

For $\ell\in\{1,y,z\}$ and $t\in[t_k,t_{k+1})$, define
\begin{gather*}
    \mathcal G_\theta^{\ell,\pi}(t,s)
    \coloneqq
    \mathbf 1_{[t_{k+1},T]}(s)
    \frac{1}{\Delta_k}
    \int_{t_k}^{t_{k+1}}
    \mathcal G_\theta^\ell(a,s)\,da.
\end{gather*}
These fields provide the required pointwise bounds for
$D_\theta G^\pi$, $D_\theta\partial_yG^\pi$, and
$D_\theta\partial_zG^\pi$. Moreover,
\begin{gather*}
    \gamma_1^\Pi\leq\gamma_1,
    \qquad
    \gamma_y^\Pi\leq\gamma_y,
    \qquad
    \gamma_z^\Pi\leq\gamma_z.
\end{gather*}

Finally, for $\theta<\theta'$ and $t\in[t_k,t_{k+1})$, set
\begin{gather*}
    \mathcal K_{\theta,\theta'}^\pi(t,s)
    \coloneqq
    \mathbf 1_{[t_{k+1},T]}(s)
    \frac{1}{\Delta_k}
    \int_{t_k}^{t_{k+1}}
    \mathcal K_{\theta,\theta'}(a,s)\,da, \quad \text{and} \quad \mathcal G_{\theta,\theta'}^{2,\pi}(t,s)
    \coloneqq
    \mathbf 1_{[t_{k+1},T]}(s)
    \frac{1}{\Delta_k}
    \int_{t_k}^{t_{k+1}}
    \mathcal G_{\theta,\theta'}^2(a,s)\,da.
\end{gather*}
Then, for every $(y,z)\in\mathbb R^2$,
\begin{gather*}
    |D_\theta G^\pi(t,s,y,z)-D_{\theta'}G^\pi(t,s,y,z)|
    \leq
    \mathcal K_{\theta,\theta'}^\pi(t,s), \quad \text{and} \quad |D_{\theta'}D_\theta G^\pi(t,s,y,z)|
    \leq
    \mathcal G_{\theta,\theta'}^{2,\pi}(t,s).
\end{gather*}
Jensen's inequality and the continuous envelope bounds give
\begin{gather*}
    \kappa^\Pi\leq\kappa,
    \qquad
    \gamma_2^\Pi\leq\gamma_2.
\end{gather*}
Thus, condition \emph{(iii)} also holds uniformly along $\Pi$, and hence
Assumption \ref{main assumption convergence rate BSVIEs, discrete setting}
is satisfied.

\medskip
\noindent\textit{Step 2: Passage to the limit and choice of representative.}

By Step 1 and Theorem \ref{thm:holder-regularity-system-BSDEs}, there exists a constant $K>0$, independent of $\pi\in\Pi$, such that, for every $s,r\in[0,T]$,
\begin{gather*}
    \mathbb E\bigg[
    \bigg(
    \sum_{k\colon t_k\leq s\wedge r}
    \Delta_k
    \big|
    Z^\pi(t_k,s)-Z^\pi(t_k,r)
    \big|^2
    \bigg)^{\frac p2}
    \bigg]
    \leq
    K\big(
    |s-r|^{\frac p2}
    +
    |\pi|^{\frac p2}
    \big).
\end{gather*}
Since $Z^\pi(\tau(t),s)-Z^\pi(\tau(t),r)$ is constant in $t$ on each cell, and the possible partial last cell is bounded by the corresponding full cell, it follows that
\begin{gather}\label{eq:uniform-continuous-holder-Z-inside-proof}
    \mathbb E\bigg[
    \bigg(
    \int_0^{s\wedge r}
    \big|
    Z^\pi(\tau(t),s)-Z^\pi(\tau(t),r)
    \big|^2\,dt
    \bigg)^{\frac p2}
    \bigg]
    \leq
    K\big(
    |s-r|^{\frac p2}
    +
    |\pi|^{\frac p2}
    \big).
\end{gather}
By Lemma \ref{lemma:convergence-cell-average-data},
\begin{gather*}
    \mathbb E\bigg[
    \int_0^T\int_0^u
    |Z^\pi(\tau(t),u)-Z(t,u)|^2\,dt\,du
    \bigg]
    \longrightarrow0.
\end{gather*}
Hence the non-negative functions
\begin{gather*}
    u\longmapsto
    \mathbb E\bigg[
    \int_0^u
    |Z^\pi(\tau(t),u)-Z(t,u)|^2\,dt
    \bigg]
\end{gather*}
converge to zero in $L^1(0,T)$. Passing to a subsequence, still denoted by $\pi$, we may therefore assume that, for a.e. $u\in[0,T]$,
\begin{gather}\label{eq:slice-convergence-Z}
    \mathbb E\bigg[
    \int_0^u
    |Z^\pi(\tau(t),u)-Z(t,u)|^2\,dt
    \bigg]
    \longrightarrow0.
\end{gather}

Let $s$ and $r$ be such that \eqref{eq:slice-convergence-Z} holds. The reverse triangle inequality in $L^2(0,s\wedge r)$ gives
\begin{flalign*}
    \Bigg|
    \bigg(
    \int_0^{s\wedge r}
    |Z^\pi(\tau(t),s)-Z^\pi(\tau(t),r)|^2\,dt
    &\bigg)^{\frac12}
    -
    \bigg(
    \int_0^{s\wedge r}
    |Z(t,s)-Z(t,r)|^2\,dt
    \bigg)^{\frac12}
    \Bigg|
    \\ &\leq
    \bigg(
    \int_0^{s\wedge r}
    |Z^\pi(\tau(t),s)-Z(t,s)|^2\,dt
    \bigg)^{\frac12}
    +
    \bigg(
    \int_0^{s\wedge r}
    |Z^\pi(\tau(t),r)-Z(t,r)|^2\,dt
    \bigg)^{\frac12}.
\end{flalign*}
By \eqref{eq:slice-convergence-Z}, the right-hand side converges to zero in $L^2(\Omega)$. Passing to a further subsequence if necessary, the convergence holds $\mathbb P$-a.s. Fatou's lemma and \eqref{eq:uniform-continuous-holder-Z-inside-proof} then yield
\begin{flalign*}
    \mathbb E\bigg[
    \bigg(
    \int_0^{s\wedge r}
    |Z(t,s)-Z(t,r)|^2\,dt
    \bigg)^{\frac p2}
    \bigg]
    \leq
    \liminf_{|\pi|\to0}
    \mathbb E\bigg[
    \bigg(
    \int_0^{s\wedge r}
    |Z^\pi(\tau(t),s)-Z^\pi(\tau(t),r)|^2\,dt
    \bigg)^{\frac p2}
    \bigg]
    \leq
    K|s-r|^{\frac p2}.
\end{flalign*}
Therefore,
\begin{gather}\label{eq:ae-holder-Z}
    \mathbb E\bigg[
    \bigg(
    \int_0^{s\wedge r}
    |Z(t,s)-Z(t,r)|^2\,dt
    \bigg)^{\frac p2}
    \bigg]
    \leq
    K|s-r|^{\frac p2}
\end{gather}
for a.e. $(s,r)\in[0,T]^2$.

It remains to choose a representative for which the estimate holds for every $s,r\in[0,T]$. Fix $a\in[0,T)$ and set
\begin{gather*}
    E_a
    \coloneqq
    L^p(\Omega;L^2(0,a)).
\end{gather*}
Consider the map
\begin{gather*}
    [a,T]\ni u
    \longmapsto
    Z^a(u)
    \coloneqq
    \{Z(t,u)\}_{t\in[0,a]},
\end{gather*}
which takes values in $E_a$ for a.e. $u$. By \eqref{eq:ae-holder-Z},
\begin{gather*}
    \|Z^a(u)-Z^a(v)\|_{E_a}^p
    =
    \mathbb E\bigg[
    \bigg(
    \int_0^a
    |Z(t,u)-Z(t,v)|^2\,dt
    \bigg)^{\frac p2}
    \bigg]
    \leq
    K|u-v|^{\frac p2}
\end{gather*}
for a.e. $(u,v)\in[a,T]^2$. Since $E_a$ is complete, the usual extension argument provides a representative, still denoted by $Z^a$, satisfying
\begin{gather}\label{eq:holder-version-fixed-a}
    \|Z^a(u)-Z^a(v)\|_{E_a}^p
    \leq
    K|u-v|^{\frac p2},
    \qquad
    u,v\in[a,T].
\end{gather}

Apply this construction for every rational $a\in[0,T)$. If $0\leq b<a<T$ are rational, the restriction of $Z^a$ to $L^p(\Omega;L^2(0,b))$ and $Z^b$ are continuous maps on $[a,T]$ which coincide for a.e. $u$. They therefore coincide for every $u\in[a,T]$. Hence these representatives define a single representative of $Z$ for which \eqref{eq:holder-version-fixed-a} holds for every rational $a\in[0,T)$.

Finally, let $s,r\in[0,T]$. If $s\wedge r=0$, the estimate is immediate. Otherwise, choose rational numbers $a_m\uparrow s\wedge r$. Since $s,r\in[a_m,T]$, \eqref{eq:holder-version-fixed-a} gives
\begin{gather*}
    \mathbb E\bigg[
    \bigg(
    \int_0^{a_m}
    |Z(t,s)-Z(t,r)|^2\,dt
    \bigg)^{\frac p2}
    \bigg]
    \leq
    K|s-r|^{\frac p2}.
\end{gather*}
Letting $m\to\infty$ and applying the monotone convergence theorem concludes the proof.

\subsection{Proof of Corollary \ref{cor:holder-continuity-BSVIE}}

For every $n \in \mathbb N$, let us define the $n$-dyadic partition
\[
    \mathcal D_n
    \coloneqq
    \big\{
    t_j^n\coloneqq j \, h_n
    :
    j=0,\dots,2^n
    \big\},
    \qquad
    h_n\coloneqq T2^{-n}.
\]
For $j=0,\dots,2^n-1$, set
\begin{gather}\label{eq:K_n}
    \xi_{n,j}
    \coloneqq
    \bigg(
    \int_0^{t_j^n}
    \big|
    Z(t,t_{j+1}^n)-Z(t,t_j^n)
    \big|^2\,dt
    \bigg)^{\frac12} \qquad \text{and}
    \qquad
    K_n
    \coloneqq
    \max_{0\leq j\leq2^n-1}\xi_{n,j}.
\end{gather}
By Theorem \ref{thm:holder-regularity-BSVIE}, we have $\mathbb{E}|\xi_{n,j}|^p\leq Kh_n^{\frac p2}$. Since $K_n^p\leq\sum_{j=0}^{2^n-1}\xi_{n,j}^p$, it follows that $\mathbb{E}\big[|K_n|^p\big]^{\frac{1}{p}}\leq C h_n^{\frac12-\frac1p}$. Hence, for every $\gamma<\frac12-\frac1p$, we have that
\begin{gather}\label{eq:Cgamma}
   \mathbb{E}\big[ |C_\gamma'|^{p} \big]^{\frac{1}{p}}
    \leq
    C\sum_{n=0}^{\infty}
    h_n^{\frac12-\frac1p-\gamma}
    <
    \infty, \quad \text{where} \quad C_\gamma'
    \coloneqq
    \sum_{n=0}^{\infty}
    h_n^{-\gamma}K_n.
\end{gather}
We now apply the classical dyadic chaining argument. Let $s,r\in \mathcal D\coloneqq\bigcup_{n\geq0}\mathcal D_n$ with $s<r$, and choose
$m\geq0$ such that
\[
    h_{m+1}<r-s\leq h_m.
\]
For $n\geq m$, let $s_n$ and $r_n$ be the smallest points of
$\mathcal D_n$ greater than or equal to $s$ and $r$, respectively. Then
$s_n\downarrow s$ and $r_n\downarrow r$, and both sequences are eventually
constant. Moreover, $s_m$ and $r_m$ are either equal or adjacent in
$\mathcal D_m$, while, for every $n\geq m$, each of the pairs
$(s_n,s_{n+1})$ and $(r_n,r_{n+1})$ consists either of the same point or
of two adjacent points of $\mathcal D_{n+1}$.

The smaller endpoint of every dyadic pair appearing below is greater than
or equal to $s$. Hence, the corresponding estimate in
\eqref{eq:K_n} may be restricted to the common integration interval
$[0,s]$. Since the two sequences are eventually constant, the following
identity holds in $L^2(0,s)$:
\begin{flalign*}
    Z(\cdot,r)-Z(\cdot,s)
    &=
    \sum_{n=m}^{\infty}
    \big(
    Z(\cdot,r_{n+1})-Z(\cdot,r_n)
    \big)
    +
    \big(
    Z(\cdot,r_m)-Z(\cdot,s_m)
    \big)
    +
    \sum_{n=m}^{\infty}
    \big(
    Z(\cdot,s_n)-Z(\cdot,s_{n+1})
    \big).
\end{flalign*}
Therefore, by the triangle inequality in $L^2(0,s)$, the definition of
$K_n$ and $h_m=2h_{m+1}<2(r-s)$, one has that, outside a fixed null set,
\begin{gather}\label{eq:dyadic-holder-all-pairs}
    \bigg(
    \int_0^{s\wedge r}
    |Z(t,s)-Z(t,r)|^2\,dt
    \bigg)^{\frac12}
    \leq
    2^{\gamma+1}C_\gamma'|s-r|^\gamma,
    \qquad
    s,r\in\mathcal D,
\end{gather}

For $u\in[0,T]$, let $u_n\downarrow u$ be its right dyadic approximation. By
\eqref{eq:dyadic-holder-all-pairs}, $\{Z(\cdot,u_n)\}_{n\geq0}$ is almost surely Cauchy in $L^2(0,u)$, and denote its limit by $\widetilde Z(\cdot,u)$. The usual separability argument provides a jointly measurable representative. Moreover, Theorem \ref{thm:holder-regularity-BSVIE} implies that $Z(\cdot,u_n)\to Z(\cdot,u)$ in probability in $L^2(0,u)$. Hence, by
uniqueness of limits in probability, $\widetilde Z$ is a version of $Z$.

Finally, approximating arbitrary $s,r\in[0,T]$ by right dyadic points and
passing to the limit in \eqref{eq:dyadic-holder-all-pairs}, we obtain
\[
    \bigg(
    \int_0^{s\wedge r}
    |\widetilde Z(t,s)-\widetilde Z(t,r)|^2\,dt
    \bigg)^{\frac12}
    \leq
    2^{\gamma+1}C_\gamma'|s-r|^\gamma.
\]
Setting $C_\gamma\coloneqq2^{\gamma+1}C_\gamma'$ concludes the proof.

\subsection{Examples}\label{section:examples}

We conclude the section with two examples of situations in which Assumption \ref{assumption-Holder-BSVIE} is satisfied. The proofs are given in \ref{section:proof-example}.

\begin{example}\label{example-1}
Assume that $(\Psi, G)$ are such that:
\begin{itemize}
    \item $G \colon \Delta(0,T) \times \mathbb{R} \times \mathbb{R} \to \mathbb{R}$ is a deterministic function with continuous and uniformly bounded first- and second-order partial derivatives with respect to $y$ and $z$, and
    \begin{gather*}
        \int_0^T \int_t^T |G(t,s,0,0)|^2 ds dt < \infty;
    \end{gather*}
    \item The terminal datum $\Psi = (\Psi(t))_{t \in [0,T]}$ is given by a time-varying chaos expansion of finite order, i.e.
    \begin{gather*}
        \Psi(t)
        =
        \sum_{k=0}^{M}
        \int_{[0,T]^{k}}
        g_k(s_1, \dots, s_k, t)
        \, dB(s_1) \cdots dB(s_k),
    \end{gather*}
    where the integrals are understood as multiple Wiener integrals. For each $k \in \{0, \dots, M\}$, $g_k \in L^2([0,T]^{k+1})$ is symmetric in the first $k$ variables. We assume that the kernels satisfy the following conditions. First, there exists a constant $C>0$, independent of $\theta,\theta'$, such that, for all $\theta,\theta'\in[0,T]$,
\begin{gather}\label{eq:example-1-p1-kernel}
    \sum_{k=1}^{M}
    \int_0^T
    \big\|
    g_k(\cdot,\theta,t)-g_k(\cdot,\theta',t)
    \big\|_{L^2([0,T]^{k-1})}^2\,dt
    \leq
    C|\theta-\theta'|.
\end{gather}
Moreover,
\begin{gather}\label{eq:example-1-p3-kernel}
    \sup_{\theta,\theta'\in[0,T]}
    \sum_{k=2}^{M}
    \int_0^T
    \big\|
    g_k(\cdot,\theta,\theta',t)
    \big\|_{L^2([0,T]^{k-2})}^2\,dt
    <\infty,
\end{gather}
and
\begin{flalign}\label{eq:example-1-eq-1}
    \sup_{\theta,t\in[0,T]}
    \sum_{k=1}^{M}
    \big\|
    g_k(\cdot,\theta,t)
    \big\|_{L^2([0,T]^{k-1})}^2
    &<\infty.
\end{flalign}
Finally, for some $\eta>1/q$, there exists a constant $C>0$, independent of $\theta,t$ and $r$, such that
\begin{gather}\label{eq:example-1-eq-2}
    \sup_{\theta\in[0,T]}
    \sum_{k=1}^{M}
    \big\|
    g_k(\cdot,\theta,t)-g_k(\cdot,\theta,r)
    \big\|_{L^2([0,T]^{k-1})}
    \leq
    C|t-r|^\eta,
    \quad t,r\in[0,T].
\end{gather}
\end{itemize}
\end{example}

\begin{example}[Forward--backward system]\label{example-2}
    Let $X(0)=x_0\in\mathbb R$ and consider the decoupled system
    \begin{flalign*}
        X(t) &= X(0) + \int_0^t b(s, X(s)) ds + \int_0^t \sigma(s, X(s)) dB(s) \\
        Y(t) &= \psi\big(t,\varphi(t),\varphi(T)\big) + \int_t^T g(t, s, X(t), X(s), Y(s), Z(t,s)) ds - \int_t^T Z(t,s) dB(s),
    \end{flalign*}
    with 
    \begin{gather*}
        \varphi(t) \coloneqq \int_0^t h(r,X(r))\,dr.
    \end{gather*}
    We make the following assumptions:
    \begin{itemize}
    \item $b,\sigma$ are deterministic functions, measurable in $t$ and twice continuously differentiable with respect to $x$, with uniformly bounded first- and second-order partial derivatives with respect to $x$. Moreover,
\begin{gather*}
    \sup_{t\in[0,T]}
    \big(
    |b(t,0)|+|\sigma(t,0)|
    \big)
    <\infty.
\end{gather*}
    \item there exists a constant $C>0$ such that, for every $s,t\in[0,T]$ and $x\in\mathbb R$,
    \begin{flalign*}
        |\sigma(t,x)-\sigma(s,x)|
        \leq
        C|t-s|^{\frac12};
    \end{flalign*}
    \item $h:[0,T]\times\mathbb R\to\mathbb R$ is measurable in the time variable and twice continuously differentiable in the space variable. Moreover, $h$, $\partial_xh$, and $\partial_{xx}h$ have polynomial growth in $x$, uniformly in $t$;

    \item $\psi:[0,T]\times\mathbb R^2\to\mathbb R$ is measurable in the first time variable and twice continuously differentiable in the last two variables. Moreover, $\psi$ and its first and second-order partial derivatives with respect to the last two variables have polynomial growth, uniformly in $t$;
    
    \item $g\colon\Delta(0,T)\times\mathbb R^4\to\mathbb R$ is measurable in $(t,s)$ and twice continuously differentiable in $(x_1,x_2,y,z)$. Moreover, all first and second-order partial derivatives of $g$ with respect to $(x_1,x_2,y,z)$ are uniformly bounded in $(t,s,x_1,x_2,y,z)$, and
\begin{gather*}
    \int_0^T\int_t^T
    |g(t,s,0,0,0,0)|^2\,ds\,dt
    <\infty.
\end{gather*}
\end{itemize}
The same argument also covers generators with a similar path-dependent structure, for instance generators depending on integral functionals of the forward process, as in the terminal condition above. Similarly, the terminal datum may depend directly on $X(t)$ and $X(T)$, for example through
\[
    \psi\big(t,X(t),X(T),\varphi(t),\varphi(T)\big),
\]
provided the corresponding derivatives satisfy analogous polynomial-growth conditions.
\end{example}

\section{Numerical approximation of the BSVIE}\label{section:numerical-approx}

In this section we propose an explicit Euler-type scheme for the approximating BSDE system and use the regularity estimates from the previous sections to obtain a convergence rate.

\subsection{Left-point approximation of the BSVIE data}

The most direct way to construct the discrete data is to evaluate the BSVIE coefficients at the grid points. This choice is more convenient from a numerical point of view than the cell-average approximation, since the latter is usually not directly implementable.

More precisely, given the data $(\Psi,G)$ and a partition $\pi$ of $[0,T]$, we approximate the terminal datum by left-point evaluation and the generator by left-point evaluation away from the diagonal cell, where it is set equal to zero. For each $0\leq k\leq N-1$, define
\begin{gather}\label{eq:left-point-data}
    \Psi_k^\pi \coloneqq \Psi(t_k),
    \qquad
    G_k^\pi(s,y,z)
    \coloneqq
    \begin{cases}
    0, & s\in[t_k,t_{k+1}), \\[0.4em]
    G(t_k,s,y,z), & s\in[t_{k+1},T].
    \end{cases}
\end{gather}
The convergence rate for the left-point approximation will be obtained under the following continuity assumption on the data in the time variables.
\begin{assumption}\label{assumption:continuity-nodal-data}
The assumptions of Theorem \ref{thm:well-posedness-BSVIE} hold. Moreover, there exists an increasing continuous function
$\rho\colon[0,\infty)\to[0,\infty)$, with $\rho(0)=0$, such that, for every
$(t,s),(t',s')\in\Delta(0,T)$ and every $(X_1 ,X_2)\in L^2(\Omega)\times L^2(\Omega)$,
\begin{flalign*}
    \mathbb{E} \Big[
    &\big| \Psi(t) - \Psi(t')\big|^2
     +
    \big| G(t, s, X_1, X_2) - G(t', s', X_1, X_2)\big|^2
    \Big]^{\frac{1}{2}}
    \\ &\leq
    \Big(
    \rho\big(|t-t'|\big) + \rho\big(|s-s'|\big)
    \Big)
    \Big\{
    \norm{\Psi}_{2,2}
    +
    \norm{G^0}_{\mathcal{H}_\Delta^2} 
    +
    \mathbb{E}\big[ |X_1|^2 \big]^{\frac{1}{2}}
    +
    \mathbb{E}\big[ |X_2|^2 \big]^{\frac{1}{2}}
    \Big\}. 
\end{flalign*}
\end{assumption}

We now show that the BSDE system associated with the left-point data converges to the original BSVIE with rate controlled by the modulus $\rho$.

\begin{lemma}\label{lemma:convergence-left-point-data}
    Let Assumption \ref{assumption:continuity-nodal-data} hold, and let $(Y,Z)$ be the unique solution to the BSVIE \eqref{eq:bsvie-intro}. Let $(\Psi^\pi,G^\pi)$ be the discrete data given by \eqref{eq:left-point-data}. Then the BSDE system \eqref{def:BSDE-app-BSVIE} admits a unique solution $(Y^\pi,Z^\pi)$. Moreover, there exists a constant $C>0$, independent of $\pi$, such that
    \begin{flalign*}
        \mathbb E\Bigg[
        \int_0^T |Y(t)-Y^\pi(\tau(t),t)|^2\,dt
        +
        \int_0^T\int_t^T
        |Z(t,s)-Z^\pi(\tau(t),s)|^2\,ds\,dt
        \Bigg]
        \leq
        C\big(\rho(|\pi|)^2+|\pi|\big).
    \end{flalign*}
\end{lemma}

\begin{proof}
The discrete data \eqref{eq:left-point-data} satisfy the assumptions of Theorem \ref{theorem:energy-estimates}, with Lipschitz constant bounded by that of $G$. Hence the BSDE system \eqref{def:BSDE-app-BSVIE} admits a unique solution.

By Theorem \ref{thm:convergence-BSDE-system-BSVIE}, it is enough to estimate $\varepsilon_\Psi^\pi$ and $\varepsilon_G^\pi$. The terminal error is immediate from Assumption \ref{assumption:continuity-nodal-data}:
\[
    \varepsilon_\Psi^\pi
    =
    \mathbb E\int_0^T
    |\Psi(t)-\Psi(\tau(t))|^2\,dt
    \leq
    C\rho(|\pi|)^2.
\]
We now estimate the generator error. Let $\tau^{\star}(t)$ denote the right endpoint of the cell containing $t$. Since the left-point generator vanishes on the diagonal cell, using $(a+b)^2\leq2a^2+2b^2$, the generator error is bounded by the sum of the off-diagonal contribution and the diagonal-strip contribution.

For the off-diagonal part, Cauchy's inequality and Assumption \ref{assumption:continuity-nodal-data} give
\[
    \mathbb E\int_0^T
    \left(
    \int_{\tau^{\star}(t)}^T
    \big|
    G(t,s,Y(s),Z(t,s))
    -
    G(\tau(t),s,Y(s),Z(t,s))
    \big|\,ds
    \right)^2dt
    \leq
    C\rho(|\pi|)^2,
\]
where we used the energy estimate for the BSVIE solution. For the diagonal strip, again by Cauchy's inequality,
\begin{flalign*}
    \mathbb E\int_0^T
    \left(
    \int_t^{\tau^{\star}(t)}
    |G(t,s,Y(s),Z(t,s))|\,ds
    \right)^2dt
    \leq
    |\pi|\,
    \mathbb E\int_0^T\int_t^{\tau^{\star}(t)}
    |G(t,s,Y(s),Z(t,s))|^2\,ds\,dt.
\end{flalign*}
The last term is bounded by $C|\pi|$, using the Lipschitz property of $G$, the integrability of $G(\cdot,\cdot,0,0)$, and the BSVIE energy estimate. Therefore,
\[
    \varepsilon_G^\pi
    \leq
    C\big(\rho(|\pi|)^2+|\pi|\big).
\]
The desired estimate now follows from Theorem \ref{thm:convergence-BSDE-system-BSVIE}.
\end{proof}

\subsection{An explicit Euler scheme}

We now introduce an explicit Euler-type approximation of the BSDE system \eqref{def:BSDE-app-BSVIE}, inspired by the scheme in \cite{zhang_bsde_method}. Fix a partition $\pi$. For each $k\in\{0,\dots,N-1\}$, set
\begin{flalign*}
    \mathcal Y^\pi(t_k,t_N)=\Psi_k^\pi.
\end{flalign*}
Then, for $l\in\{k,\dots,N-1\}$ and $s\in[t_l,t_{l+1})$, define
\begin{flalign}
    \mathcal{Y}^{\pi}(t_{k}, s)
    =
    \mathcal{Y}^{\pi}(t_k, t_{l+1})
    +
    \mathbf 1_{\{l\geq k+1\}}
    \Delta_l
    G_{k}^{\pi} \Big(
    t_{l+1},
    \mathcal{Y}^{\pi}&(t_{l}, t_{l+1}),
    \scalar{\mathcal{Z}^{\pi}}_{k,l+1}
    \Big) 
    -
    \int_s^{t_{l+1}}
    \mathcal{Z}^\pi(t_k, r)\,dB(r),
    \label{eq:euler-scheme}
\end{flalign}
where, for $l\in\{k,\dots,N-1\}$,
\begin{flalign*}
    \scalar{\mathcal{Z}^{\pi}}_{k,l}
    \coloneqq
    \frac{1}{\Delta_l}
    \mathbb{E}_{t_l}
    \bigg[
    \int_{t_l}^{t_{l+1}}
    \mathcal{Z}^\pi(t_k, r)\,dr
    \bigg].
\end{flalign*}
We use the convention $\scalar{\mathcal{Z}^\pi}_{k,N}=0$. The indicator $\mathbf 1_{\{l\geq k+1\}}$ implements the convention that the generator vanishes on the diagonal cell. The use of the next-cell conditional average $\scalar{\mathcal{Z}^{\pi}}_{k,l+1}$ in the $Z$-variable makes the scheme explicit.

\begin{remark}
Let
\begin{flalign*}
    F_{k,l}^\pi
    \coloneqq
    \mathcal{Y}^{\pi}(t_k,t_{l+1})
    +
    \mathbf 1_{\{l\geq k+1\}}
    \Delta_l
    G_k^\pi
    \Big(
    t_{l+1},
    \mathcal{Y}^{\pi}(t_l,t_{l+1}),
    \scalar{\mathcal Z^\pi}_{k,l+1}
    \Big).
\end{flalign*}
Then $F_{k,l}^\pi$ is $\mathcal F_{t_{l+1}}$-measurable, and \eqref{eq:euler-scheme} can be written equivalently as
\begin{flalign*}
    \mathcal Y^\pi(t_k,s)
    =
    \mathbb E_s\big[F_{k,l}^\pi\big],
    \qquad
    s\in[t_l,t_{l+1}).
\end{flalign*}
Moreover,
\begin{flalign*}
    \scalar{\mathcal Z^\pi}_{k,l}
    =
    \frac{1}{\Delta_l}
    \mathbb E_{t_l}
    \big[
    F_{k,l}^\pi (B(t_{l+1})-B(t_l))
    \big].
\end{flalign*}
This conditional-expectation representation of
$(\mathcal Y^\pi,\scalar{\mathcal Z^\pi})$ is standard in Euler-type schemes for BSDEs.
\end{remark}

We measure the error between the BSDE system \eqref{def:BSDE-app-BSVIE} and the Euler scheme \eqref{eq:euler-scheme} by
\begin{flalign*}
    \mathcal{E}^{Y}(\pi)
    &\coloneqq
    \mathbb{E}\bigg[
    \int_0^T
    \big|
    Y^{\pi}(\tau(t), t)-\mathcal{Y}^\pi(\tau(t), t)
    \big|^2\,dt
    \bigg],
    \\
    \mathcal{E}^{Z}(\pi)
    &\coloneqq
    \mathbb{E}\bigg[
    \int_0^T
    \int_t^T
    \big|
    Z^\pi(\tau(t), s)-\mathcal{Z}^{\pi}(\tau(t), s)
    \big|^2\,ds\,dt
    \bigg].
\end{flalign*}
For $0\leq k\leq l\leq N-1$, we also define the conditional cell averages of the exact martingale integrand by
\begin{flalign*}
    \scalar{Z^\pi}_{k,l}
    \coloneqq
    \frac{1}{\Delta_l}
    \mathbb E_{t_l}
    \bigg[
    \int_{t_l}^{t_{l+1}}
    Z^\pi(t_k,r)\,dr
    \bigg].
\end{flalign*}
These quantities are well defined independently of the choice of representative of $Z^\pi(t_k,\cdot)$.

The next result shows that the Euler error is controlled by the diagonal-cell contribution and by the time-regularity of the martingale component $Z^\pi$ along the partition. The proof follows the classical Euler scheme argument for BSDEs and is deferred to \ref{section:proof-euler}.

\begin{theorem}\label{thm:euler-error-bsde-system}
Let Assumption \ref{assumption:continuity-nodal-data} hold, and let $(\Psi^\pi,G^\pi)$ be the left-point data defined in \eqref{eq:left-point-data}. Assume also that there exists a constant $L>0$, independent of $\pi$, such that $\max_{0\leq k\leq N-2} \Delta_k/\Delta_{k+1}\leq L$.
Then there exists a constant $C>0$, independent of $\pi$, such that, for $|\pi|$ sufficiently small,
\begin{flalign*}
    \mathcal E^Y(\pi)+\mathcal E^Z(\pi)
    &\leq
    C\bigg(
    \rho(|\pi|)^2
    +
    |\pi|
    +
    \sum_{k=0}^{N-1}
    \Delta_k
    \mathbb E\bigg[
    \int_{t_k}^{t_{k+1}}
    |Z^\pi(t_k,s)|^2\,ds
    \bigg]
    \\
    &\quad+
    \sum_{k=0}^{N-2}
    \Delta_k
    \sum_{l=k+1}^{N-2}
    \mathbb E\bigg[
    \int_{t_l}^{t_{l+1}}
    \big|
    Z^\pi(t_k,s)-\scalar{Z^\pi}_{k,l+1}
    \big|^2\,ds
    \bigg]
    \bigg).
\end{flalign*}
\end{theorem}

The following assumption provides the pointwise-in-time regularity on the data $(\Psi,G)$ needed to apply the discrete Hölder estimate in Theorem \ref{thm:holder-regularity-system-BSDEs} to the left-point approximation. This estimate will then be used to control the off-diagonal $Z^\pi$-oscillation terms in Theorem \ref{thm:euler-error-bsde-system}, and hence to obtain a convergence rate for the explicit Euler scheme.

\begin{assumption}\label{assumption:Holder-BSVIE-nodal}
Let $(\Psi,G)$ satisfy Assumption \ref{assumption-Holder-BSVIE} with $p=2$. We further impose the following pointwise-in-time assumptions.

We assume that $D_\theta\Psi$, $\mathcal G_\theta^1$ and $\mathcal G_\theta^z$ admit jointly measurable versions for which the quantities defining $\mathfrak p_2$, $\gamma_1$ and $\gamma_z$ in Assumption \ref{assumption-Holder-BSVIE} remain finite when each essential supremum appearing in its definition is replaced by an ordinary supremum. These versions are used in the left-point discretization.

\begin{enumerate}[label=(\roman*)]
    \item The free term admits a version satisfying
    $\sup_{t\in[0,T]}\mathbb E|\Psi(t)|^q<\infty$, such that
    $\Psi(t)\in\mathbb D^{2,2}$ for every $t\in[0,T]$, and for which
    \begin{flalign*}
        \mathfrak p_1^{\mathrm{l}}
        &\coloneqq
        \sup_{\theta<\theta'}
        \sup_{t\in[0,T]}
        \frac{
        \mathbb E
        |D_\theta\Psi(t)-D_{\theta'}\Psi(t)|^2
        }{
        |\theta-\theta'|
        +
        \mathbf 1_{\{t\in[\theta,\theta']\}}
        }
        <\infty,
        \qquad
        \mathfrak p_3^{\mathrm{l}}
        \coloneqq
        \sup_{\theta,\theta'\in[0,T]}
        \sup_{t\in[0,T]}
        \mathbb E
        |D_{\theta'}D_\theta\Psi(t)|^2
        <\infty.
    \end{flalign*}

    \item We assume that
    \begin{flalign*}
        \mathfrak g_0^{\mathrm{l}}
        \coloneqq
        \sup_{t\in[0,T]}
        \mathbb E\bigg[
        \bigg(
        \int_t^T
        |G(t,s,0,0)|^2\,ds
        \bigg)^{\frac q2}
        \bigg]
        <\infty.
    \end{flalign*}

   \item Let $\mathcal G_\theta^y$, $\mathcal K_{\theta,\theta'}$, and
$\mathcal G_{\theta,\theta'}^{2}$ be the random fields appearing in
Assumption \ref{assumption-Holder-BSVIE}. We assume that
\begin{flalign*}
    \gamma_y^{\mathrm{l}}
    &\coloneqq
    \sup_{\theta\in[0,T]}
    \mathbb E\bigg[
    \sup_{(t,s)\in\Delta(0,T)}
    |\mathcal G_\theta^y(t,s)|^q
    \bigg]^{1/q}
    <\infty, \quad \kappa^{\mathrm{l}}
    \coloneqq
    \sup_{\theta<\theta'}
    \sup_{t\in[0,T]}
    \frac{
    \displaystyle
    \mathbb E\bigg[
    \int_t^T
    |\mathcal K_{\theta,\theta'}(t,s)|^2\,ds
    \bigg]
    }{
    |\theta-\theta'|
    +
    \mathbf 1_{\{t\in[\theta,\theta']\}}
    }
    <\infty,
    \\
    \gamma_2^{\mathrm{l}}
    &\coloneqq
    \sup_{\theta,\theta'\in[0,T]}
    \sup_{t\in[0,T]}
    \mathbb E
    \bigg[
    \int_t^T
    |\mathcal G_{\theta,\theta'}^{2}(t,s)|^2\,ds
    \bigg]
    <\infty.
\end{flalign*}
\end{enumerate}
\end{assumption}

\begin{remark}
For Example \ref{example-1}, Assumptions \ref{assumption:continuity-nodal-data} and \ref{assumption:Holder-BSVIE-nodal} reduce to pointwise-in-time and time-continuity bounds for the kernels in the chaos expansion of $\Psi$, together with the required time-continuity of the deterministic generator.

For Example \ref{example-2}, the pointwise-in-time Malliavin bounds follow from standard moment estimates for the forward process and its Malliavin derivatives. The only additional requirement is the time-continuity condition in Assumption \ref{assumption:continuity-nodal-data}, which follows from suitable modulus of continuity of $\psi$ and $g$ in the time variables, together with the standard time regularity of $X$ and $\varphi$.
\end{remark}

Combining the convergence estimate for the left-point BSDE system with Theorem \ref{thm:euler-error-bsde-system}, we obtain a convergence rate for the discrete approximation of the original BSVIE.

\begin{theorem}\label{thm:euler-convergence-bsvie}
Let $(\Psi,G)$ satisfy Assumptions \ref{assumption:continuity-nodal-data} and \ref{assumption:Holder-BSVIE-nodal}, and assume that the partitions satisfy the mesh-ratio condition in Theorem \ref{thm:euler-error-bsde-system}. Then there exists a constant $C>0$, independent of $\pi$, such that, for $|\pi|$ sufficiently small,
\begin{flalign*}
\mathbb E\Bigg[
\int_0^T |Y(t)-\mathcal Y^\pi(\tau(t),t)|^2\,dt
+
\int_0^T\int_t^T
|Z(t,s)-\mathcal Z^\pi&(\tau(t),s)|^2\,ds\,dt
\Bigg]
\leq
C\big(
\rho(|\pi|)^2+|\pi|
\big).
\end{flalign*}
\end{theorem}

\begin{proof}
The pointwise bounds in Assumption \ref{assumption:Holder-BSVIE-nodal} imply that the left-point data \eqref{eq:left-point-data} satisfy Assumption \ref{main assumption convergence rate BSVIEs, discrete setting}, with $p=2$, uniformly in $\pi$. Hence Theorem \ref{thm:holder-regularity-system-BSDEs} and Proposition \ref{prop:malliavin-energy-estimates-discrete-BSVIE} apply with constants independent of $\pi$.

We first control the diagonal-cell term in Theorem \ref{thm:euler-error-bsde-system}. By the trace identity \eqref{eq:malliavin-trace-bsde-system}, it is enough to show that
\[
    \mathbb E\int_0^T
    |D_sY^\pi(\tau(s),s)|^2\,ds
    \leq C.
\]
Applying the standard $L^2$ estimate to the Malliavin-differentiated BSDE \eqref{eq:DY-discrete-BSVIE}, with Malliavin parameter $s$, and then integrating in $s$, this follows from the uniform bounds on $D_\theta\Psi^\pi$, the envelope defining $\gamma_1^\Pi$, and Proposition \ref{prop:malliavin-energy-estimates-discrete-BSVIE}. Therefore, again by the trace identity,
\begin{flalign}
\sum_{k=0}^{N-1}
\Delta_k
\mathbb E\int_{t_k}^{t_{k+1}}
|Z^\pi(t_k,s)|^2\,ds
=
\sum_{k=0}^{N-1}
\Delta_k
\mathbb E\int_{t_k}^{t_{k+1}}
|D_sY^\pi(t_k,s)|^2\,ds
\leq
|\pi|
\mathbb E\int_0^T
|D_sY^\pi(\tau(s),s)|^2\,ds
\leq
C|\pi|.
\label{eq:diagonal-cell-Zpi-bound}
\end{flalign}

We next estimate the conditional-average oscillation term. Inserting $Z^\pi(t_k,t_{l+1})$ and using conditional Jensen's inequality gives
\begin{flalign*}
\sum_{k=0}^{N-2}
\Delta_k
\sum_{l=k+1}^{N-2}
\mathbb E\int_{t_l}^{t_{l+1}}
\big|
Z^\pi(t_k,s)-\scalar{Z^\pi}_{k,l+1}
\big|^2\,ds
&\leq
C\sum_{k=0}^{N-2}
\Delta_k
\sum_{l=k+1}^{N-2}
\mathbb E\int_{t_l}^{t_{l+1}}
|Z^\pi(t_k,s)-Z^\pi(t_k,t_{l+1})|^2\,ds
\\
&\qquad+
C\sum_{k=0}^{N-2}
\Delta_k
\sum_{l=k+1}^{N-2}
\frac{\Delta_l}{\Delta_{l+1}}
\mathbb E\int_{t_{l+1}}^{t_{l+2}}
|Z^\pi(t_k,t_{l+1})-Z^\pi(t_k,r)|^2\,dr.
\end{flalign*}
The mesh-ratio condition and Theorem \ref{thm:holder-regularity-system-BSDEs}, with $p=2$, then yield
\begin{flalign}
&\sum_{k=0}^{N-2}
\Delta_k
\sum_{l=k+1}^{N-2}
\mathbb E\int_{t_l}^{t_{l+1}}
\big|
Z^\pi(t_k,s)-\scalar{Z^\pi}_{k,l+1}
\big|^2\,ds
\notag\\
&\quad\leq
C\sum_{l=1}^{N-2}
\int_{t_l}^{t_{l+1}}
\big(
|s-t_{l+1}|+|\pi|
\big)\,ds
+
C\sum_{l=1}^{N-2}
\int_{t_{l+1}}^{t_{l+2}}
\big(
|r-t_{l+1}|+|\pi|
\big)\,dr
\notag\\
&\quad\leq
C|\pi|.
\label{eq:conditional-average-Zpi-bound}
\end{flalign}

Theorem \ref{thm:euler-error-bsde-system}, together with
\eqref{eq:diagonal-cell-Zpi-bound} and
\eqref{eq:conditional-average-Zpi-bound}, now gives
\begin{flalign*}
\mathcal E^Y(\pi)+\mathcal E^Z(\pi)
&\leq
C\big(
\rho(|\pi|)^2+|\pi|
\big).
\end{flalign*}
Combining this estimate with Lemma \ref{lemma:convergence-left-point-data} and the triangle inequality concludes the proof.
\end{proof}

\appendix

\section{Proofs of Subsection \ref{subsec:estimates-fundamental-matrix}}\label{section:appendix-proofs-fundamental-matrix}

\begin{proof}[Proof of Lemma \ref{lemma:bound_U_delta}]
By \eqref{eq:U-diagonal-expression} and the uniform boundedness of $\beta_k^\pi$, the standard moment estimates for stochastic exponentials, together with the conditional Doob inequality, yield
\begin{gather*}
    \mathbb{E}\Big[
    \sup_{t\leq u\leq T}
    |U^\pi(t,u)_{k,k}|^\lambda
    \,\big|\,\mathcal F_t
    \Big]^{\frac{1}{\lambda}}
    \leq C_\lambda.
\end{gather*}
For $k<j$, set
\begin{gather*}
    m_{k,j}
    \coloneqq
    \mathbb{E}\Big[
    \sup_{t\leq u\leq T}
    |U^\pi(t,u)_{k,j}|^\lambda
    \,\big|\,\mathcal F_t
    \Big]^{\frac{1}{\lambda}}.
\end{gather*}
The same diagonal estimate gives
\begin{gather*}
    \mathbb{E}\Big[
    \sup_{v\leq u\leq T}
    |\mathcal E_j^\pi(v,u)|^\lambda
    \,\big|\,\mathcal F_v
    \Big]^{\frac{1}{\lambda}}
    \leq C_\lambda.
\end{gather*}
Hence, using \eqref{eq:equation-U-offdg}, conditional Minkowski's inequality and the tower property, together with the fact that the factor multiplying $\mathcal E_j^\pi(v,\cdot)$ is $\mathcal F_v$-measurable, the diagonal estimate for $\mathcal E_j^\pi$, and the uniform boundedness of $\alpha_{l,j}^\pi$, we obtain
\begin{flalign*}
    m_{k,j}
    &\leq
    C_\lambda
    \int_t^T
    \mathbf 1_{[t_j,t_{j+1})}(v)
    \sum_{l=k}^{j-1}
    \mathbb E\Big[
    |U^\pi(t,v)_{k,l}|^\lambda
    \,\big|\,\mathcal F_t
    \Big]^{\frac1\lambda}\,dv
    \leq
    C_\lambda
    \Delta_j
    \sum_{l=k}^{j-1}m_{k,l}.
\end{flalign*}
Since $m_{k,k}\leq C_\lambda$, setting
$a_j\coloneqq \Delta_j^{-1}m_{k,j}$ for $j>k$, the preceding recurrence and the discrete Gronwall inequality yield
\begin{gather*}
    a_j
    \leq
    C_\lambda
    \bigg(
    1+
    \sum_{l=k+1}^{j-1}\Delta_l a_l
    \bigg)
    \leq
    C_\lambda
    \exp\bigg(
    C_\lambda
    \sum_{l=k+1}^{j-1}\Delta_l
    \bigg)
    \leq
    C_\lambda,
    \qquad
    m_{k,j}
    =
    \Delta_j a_j
    \leq
    C_\lambda\Delta_j,
    \quad k<j.
\end{gather*}
Together with the diagonal estimate, this proves \eqref{eq:U-entry-bound}.
\end{proof}

\begin{proof}[Proof of Lemma \ref{lemma:energy-Dalpha-Dbeta}]
By the Malliavin chain rule (see \cite[Proposition 1.2.3]{NualartDavidTMCa}), the boundedness of the second-order partial derivatives of $G^\pi$, and the envelope bounds in Assumption \ref{main assumption convergence rate BSVIEs, discrete setting}, we have, for $v\in[t_{k+1},T]$,
\begin{flalign*}
    |D_u\alpha_k^\pi(v)|
    &\leq
    C\Big(
    |D_uY^\pi(\tau(v),v)|
    +
    |D_uZ^\pi(t_k,v)|
    +
    |\mathcal G_u^{y,\pi}(t_k,v)|
    \Big),
    \\
    |D_u\beta_k^\pi(v)|
    &\leq
    C\Big(
    |D_uY^\pi(\tau(v),v)|
    +
    |D_uZ^\pi(t_k,v)|
    +
    |\mathcal G_u^{z,\pi}(t_k,v)|
    \Big).
\end{flalign*}
Moreover,
\begin{gather*}
    \sum_{k=0}^{N-1}
    \Delta_k
    \int_{t_{k+1}}^T
    |D_uY^\pi(\tau(v),v)|^2\,dv
    \leq
    T\int_0^T
    |D_uY^\pi(\tau(v),v)|^2\,dv.
\end{gather*}
Since the dominating fields are piecewise constant in the first time variable and vanish on the diagonal cells,
\begin{flalign*}
    &\mathbb E\bigg[
    \bigg(
    \sum_{k=0}^{N-1}
    \Delta_k
    \int_{t_{k+1}}^T
    |\mathcal G_u^{y,\pi}(t_k,v)|^2\,dv
    \bigg)^{\frac\lambda2}
    \bigg]^{\frac1\lambda}
    \leq
    C\norm{\mathcal G_u^{y,\pi}}_{\mathcal S_\Delta^\lambda},
    \qquad \mathbb E\bigg[
    \bigg(
    \sum_{k=0}^{N-1}
    \Delta_k
    \int_{t_{k+1}}^T
    |\mathcal G_u^{z,\pi}(t_k,v)|^2\,dv
    \bigg)^{\frac\lambda2}
    \bigg]^{\frac1\lambda}
    \leq
    C\norm{\mathcal G_u^{z,\pi}}_{\mathcal U_\Delta^\lambda}.
\end{flalign*}
Combining these estimates with Proposition
\ref{prop:malliavin-energy-estimates-discrete-BSVIE} and using $\norm{\mathcal G_u^{1,\pi}}_{\mathcal H_\Delta^\lambda}
    \leq
    C\norm{\mathcal G_u^{1,\pi}}_{\mathcal T_\Delta^\lambda}$, proves the claim.
\end{proof}

\begin{proof}[Proof of Lemma \ref{lemma:diagonal-DU-bound}]
Applying the Malliavin chain rule in \cite[Proposition 3.5]{anton2024} to \eqref{eq:U-diagonal-expression}, we obtain
\begin{flalign*}
    D_u U^\pi(t,v)_{k,k}
    =
    U^\pi(t,v)_{k,k}
    \bigg(
    \int_t^v D_u\beta_k^\pi(w)\,dB(w)
    -
    \int_t^v \beta_k^\pi(w)D_u\beta_k^\pi(w)\,dw
    \bigg).
\end{flalign*}
Using conditional Hölder's inequality with conjugate exponents
$\frac{2}{2-\lambda}$ and $\frac{2}{\lambda}$, Lemma
\ref{lemma:bound_U_delta}, the conditional BDG inequality, the boundedness
of $\beta^\pi$, Cauchy--Schwarz, and the fact that $D_u\beta_k^\pi$ vanishes on
$[t_k,t_{k+1})$, we obtain
\begin{flalign*}
   \mathbb{E}_{u}\Big[
   \sup_{v \in [t,T]}
   |D_u U^\pi(t,v)_{k,k}|^{\lambda}
   \Big]^{\frac{2}{\lambda}}
   \leq
   C
   \mathbb{E}_{u}\Big[
   \int_{t_{k+1}}^T |D_u\beta_k^\pi(v)|^2\,dv
   \Big].
\end{flalign*}
Therefore, summing the previous estimate over $k$ and using Jensen's inequality for conditional expectations, we obtain
\begin{flalign*}
&\mathbb{E}\bigg[
\bigg(
\sum_{k=0}^{N-1}
\Delta_k
\mathbb{E}_{u}\Big[
\sup_{v \in [t,T]}
\big|D_u U^\pi(t,v)_{k,k}\big|^{\lambda}
\Big]^{\frac{2}{\lambda}}
\bigg)^{\frac{\eta}{2}}
\bigg]^{\frac{1}{\eta}}
\leq
C
\mathbb{E}\bigg[
\bigg(
\sum_{k=0}^{N-1}
\Delta_k
\int_{t_{k+1}}^T
|D_u\beta_k^\pi(v)|^2\,dv
\bigg)^{\frac{\eta}{2}}
\bigg]^{\frac{1}{\eta}}.
\end{flalign*}
Applying the estimate for $D_u\beta^\pi$ obtained in the proof of Lemma \ref{lemma:energy-Dalpha-Dbeta} with exponent $\eta$, and using $\eta<q$, proves the claim.
\end{proof}

\begin{proof}[Proof of Lemma \ref{lemma:bound_D_theta_off_diagonal_rho1}]
For $v\in[t,T]$, define
\begin{gather*}
    V_{k,j}(v)
    \coloneqq
    \mathbf{1}_{\{k<j\}}
    \mathbb E_u\bigg[
    \sup_{r\in[t,v]}
    |D_uU^\pi(t,r)_{k,j}|^\lambda
    \bigg]^{\frac{2}{\lambda}},
\end{gather*}
and, for $V\in\mathbb R^{N\times N}$, $\|V\|\coloneqq
\sum_{k=0}^{N-1}\Delta_k\sum_{j=k+1}^{N-1}\frac{1}{\Delta_j}|V_{k,j}|$. We also set
\begin{gather*}
    \mathcal V(v)
    \coloneqq
    \mathbb E\bigg[
    \|V(v)\|^{\frac{\eta}{2}}
    \bigg]^{\frac{2}{\eta}}.
\end{gather*}
Since the square of the left-hand side of
\eqref{ineq:bound_D_theta_off_diagonal_rho1} is $\mathcal V(T)$, it is enough to prove
\begin{gather*}
    \mathcal V(T)
    \leq
    C
    \big(
    \mathfrak p_2^\Pi
    +
    \gamma_1^\Pi
    +
    \gamma_y^\Pi
    +
    \gamma_z^\Pi
    \big)^2.
\end{gather*}
Applying the Malliavin product rule to \eqref{eq:equation-U-offdg}, and then the triangle inequality, we obtain, for $k<j$,
\begin{flalign*}
   V_{k,j}(v)
    &\leq C \Bigg\{
    \mathbb{E}_{u}\bigg[
    \sup_{w \in [t, v]}
    \bigg|
    \int_{[t,w]\cap[t_j,t_{j+1})}
    D_u\mathcal E_j^\pi(r,w)
    \sum_{i=k}^{j-1}
    \alpha_{i,j}^\pi(r)U^\pi(t,r)_{k,i}\,dr
    \bigg|^{\lambda}
    \bigg]^{\frac{2}{\lambda}}
    \\
    &\quad+
    \mathbb{E}_{u}\bigg[
    \sup_{w \in [t, v]}
    \bigg|
    \int_{[t,w]\cap[t_j,t_{j+1})}
    \mathcal E_j^\pi(r,w)
    \sum_{i=k}^{j-1}
    D_u\alpha_{i,j}^\pi(r)U^\pi(t,r)_{k,i}\,dr
    \bigg|^{\lambda}
    \bigg]^{\frac{2}{\lambda}}
    \\
    &\quad+
    \mathbb{E}_{u}\bigg[
    \sup_{w \in [t, v]}
    \bigg|
    \int_{[t,w]\cap[t_j,t_{j+1})}
    \mathcal E_j^\pi(r,w)
    \sum_{i=k}^{j-1}
    \alpha_{i,j}^\pi(r)D_uU^\pi(t,r)_{k,i}\,dr
    \bigg|^{\lambda}
    \bigg]^{\frac{2}{\lambda}}
    \Bigg\}
    \\
    &\eqqcolon
    C\Big\{
    (V_1)_{k,j}(v)
    +
    (V_2)_{k,j}(v)
    +
    (V_3)_{k,j}(v)
    \Big\}.
\end{flalign*}
Consequently,
\begin{flalign*}
    \mathcal V(v)
    \leq C \Big\{
    \mathbb{E}\bigg[
    \|V_{1}(v)\|^{\frac{\eta}{2}}
    \bigg]^{\frac{2}{\eta}}
    +
    \mathbb{E}\bigg[
    \|V_{2}(v)\|^{\frac{\eta}{2}}
    \bigg]^{\frac{2}{\eta}}
    +
    \mathbb{E}\bigg[
    \|V_{3}(v)\|^{\frac{\eta}{2}}
    \bigg]^{\frac{2}{\eta}}
    \Big\}.
\end{flalign*}

For $V_1$, conditional Minkowski and Hölder inequalities, the support and boundedness of $\alpha^\pi$, Lemma
\ref{lemma:bound_U_delta}, the diagonal estimate for
$D_u\mathcal E_j^\pi$, and conditional Jensen's inequality yield
\begin{flalign*}
    \mathbb E\bigg[
    \|V_1(v)\|^{\frac{\eta}{2}}
    \bigg]^{\frac{2}{\eta}}
    &\leq
    C
    \mathbb E\bigg[
    \bigg(
    \sum_{j=0}^{N-1}
    \Delta_j
    \int_{t_{j+1}}^T
    |D_u\beta_j^\pi(r)|^2\,dr
    \bigg)^{\frac{\eta}{2}}
    \bigg]^{\frac{2}{\eta}}.
\end{flalign*}

For $V_2$, separating the contribution $i=k$ from those with $i>k$ and using the entrywise estimates in Lemma
\ref{lemma:bound_U_delta}, the same arguments give
\begin{flalign*}
    \mathbb E\bigg[
    \|V_2(v)\|^{\frac{\eta}{2}}
    \bigg]^{\frac{2}{\eta}}
    &\leq
    C
    \mathbb E\bigg[
    \bigg(
    \sum_{i=0}^{N-1}
    \Delta_i
    \int_{t_{i+1}}^T
    |D_u\alpha_i^\pi(r)|^2\,dr
    \bigg)^{\frac{\eta}{2}}
    \bigg]^{\frac{2}{\eta}}.
\end{flalign*}

For $V_3$, the contribution corresponding to $i=k$ is controlled by
Lemma \ref{lemma:diagonal-DU-bound}. For the remaining terms, conditional
Minkowski's inequality, the support and boundedness of $\alpha^\pi$,
Cauchy--Schwarz in time, and weighted Cauchy--Schwarz in the index $i$
give
\begin{flalign*}
    (V_3^{\mathrm{off-dg}})_{k,j}(v)
    &\leq
    C\Delta_j
    \sum_{i=k+1}^{j-1}
    \frac{1}{\Delta_i}
    \int_t^v
    \mathbf{1}_{[t_j,t_{j+1})}(r)
    V_{k,i}(r)\,dr.
\end{flalign*}
Summing over $k$ and $j$ and applying Minkowski's inequality in the time
variable, we obtain
\begin{flalign*}
    \mathbb E\bigg[
    \|V_3(v)\|^{\frac{\eta}{2}}
    \bigg]^{\frac{2}{\eta}}
    &\leq
    C
    \big(
    \mathfrak p_2^\Pi
    +
    \gamma_1^\Pi
    +
    \gamma_z^\Pi
    \big)^2
    +
    C\int_t^v\mathcal V(r)\,dr.
\end{flalign*}
Combining the previous estimates with Lemma
\ref{lemma:energy-Dalpha-Dbeta}, applied with exponent $\eta$, and using
$\eta<q$, gives
\begin{gather*}
    \mathcal V(v)
    \leq
    C\bigg\{
    \big(
    \mathfrak p_2^\Pi
    +
    \gamma_1^\Pi
    +
    \gamma_y^\Pi
    +
    \gamma_z^\Pi
    \big)^2
    +
    \int_t^v\mathcal V(r)\,dr
    \bigg\}.
\end{gather*}
Gronwall's inequality therefore yields
\begin{gather*}
    \mathcal V(T)
    \leq
    C
    \big(
    \mathfrak p_2^\Pi
    +
    \gamma_1^\Pi
    +
    \gamma_y^\Pi
    +
    \gamma_z^\Pi
    \big)^2.
\end{gather*}
Taking square roots proves
\eqref{ineq:bound_D_theta_off_diagonal_rho1}.
\end{proof}

\begin{proof}[Proof of Lemma \ref{lemma:time-regularity-U}]
We first prove \eqref{eq:time-regularity-diagonal-U}. By the flow property,
\begin{gather*}
    U^\pi(r,v)_{k,k}
    =
    U^\pi(r,s)_{k,k}U^\pi(s,v)_{k,k}.
\end{gather*}
Moreover, the conditional BDG inequality, the boundedness of $\beta^\pi$, and Lemma \ref{lemma:bound_U_delta} give
\begin{flalign}
    \mathbb{E}_r\Big[
    \big|1-U^\pi(r,s)_{k,k}\big|^4
    \Big]
    &\leq
    C\mathbb{E}_r\bigg[
    \bigg(
    \int_r^s
    |U^\pi(r,a)_{k,k}|^2
    |\beta_k^\pi(a)|^2\,da
    \bigg)^2
    \bigg]
    \notag
    \\
    &\leq
    C|s-r|^2.
    \label{eq:diag-ineq-U}
\end{flalign}
Therefore, by conditional Cauchy--Schwarz and Lemma \ref{lemma:bound_U_delta},
\begin{flalign*}
    \mathbb{E}_r\Big[
    |U^\pi(s,v)_{k,k}-U^\pi(r,v)_{k,k}|^2
    \Big]
    &\leq
    \mathbb{E}_r\Big[
    |U^\pi(s,v)_{k,k}|^4
    \Big]^{\frac12}
    \mathbb{E}_r\Big[
    |1-U^\pi(r,s)_{k,k}|^4
    \Big]^{\frac12}
    \\
    &\leq
    C|s-r|,
\end{flalign*}
which proves \eqref{eq:time-regularity-diagonal-U}.

We next note that, for $k<l$, the variation-of-constants formula, Hölder's inequality in time, and Lemma \ref{lemma:bound_U_delta} yield
\begin{flalign}
    \mathbb{E}_r\Big[
    |U^\pi(r,s)_{k,l}|^4
    \Big]^{\frac12}
    &\leq
    C\delta_l(r,s)^2.
    \label{eq:useful-ineq-off-dg-U-time-reg}
\end{flalign}
Indeed, the integral defining $U^\pi(r,s)_{k,l}$ is supported on
$[r,s]\cap[t_l,t_{l+1})$, whose length is $\delta_l(r,s)$, while the integrand has uniformly bounded conditional fourth moments.

We now prove \eqref{eq:time-regularity-off-diagonal-U}. For $k<j$, the flow property gives
\begin{flalign*}
    U^\pi(s,v)_{k,j}-U^\pi(r,v)_{k,j}
    &=
    \big(1-U^\pi(r,s)_{k,k}\big)
    U^\pi(s,v)_{k,j}
    -
    U^\pi(r,s)_{k,j}
    U^\pi(s,v)_{j,j}
    -
    \sum_{l=k+1}^{j-1}
    U^\pi(r,s)_{k,l}
    U^\pi(s,v)_{l,j}.
\end{flalign*}
Using conditional Cauchy--Schwarz, \eqref{eq:diag-ineq-U},
\eqref{eq:useful-ineq-off-dg-U-time-reg}, and Lemma
\ref{lemma:bound_U_delta}, the first two terms are bounded by $C\Delta_j^2|s-r|$ and $C\delta_j(r,s)^2$, respectively. For the remaining sum, weighted Cauchy--Schwarz gives
\begin{flalign*}
    \mathbb{E}_r\bigg[
    \bigg|
    \sum_{l=k+1}^{j-1}
    U^\pi(r,s)_{k,l}
    U^\pi(s,v)_{l,j}
    \bigg|^2
    \bigg]
    \leq
    C\Delta_j^2
    \sum_{l=k+1}^{j-1}
    \frac{\delta_l(r,s)^2}{\Delta_l}
    \leq
    C\Delta_j^2
    \sum_{l=k+1}^{j-1}
    \delta_l(r,s)
    \leq
    C\Delta_j^2|s-r|.
\end{flalign*}
Combining the three bounds proves
\eqref{eq:time-regularity-off-diagonal-U}.
\end{proof}

\section{Proofs of the Examples}\label{section:proof-example}

\subsection*{Example 1}

Since $G$ is deterministic, all its Malliavin derivatives vanish. Hence the corresponding conditions in Assumption \ref{assumption-Holder-BSVIE} hold with zero envelopes. The assumptions on $G$ also give the required regularity in $(y,z)$ and integrability of $G^0$.

We now verify the conditions on the terminal datum. Since $\Psi$ belongs to a finite sum of Wiener chaoses, hypercontractivity and the Wiener isometry give
\begin{flalign*}
    \|\Psi\|_{2,p}^p
    =
    \mathbb E\bigg[
    \bigg(
    \int_0^T|\Psi(t)|^2\,dt
    \bigg)^{\frac p2}
    \bigg]
    \leq
    C
    \bigg(
    \sum_{k=0}^M
    \int_0^T
    \|g_k(\cdot,t)\|_{L^2([0,T]^k)}^2\,dt
    \bigg)^{\frac p2}
    <\infty.
\end{flalign*}
Moreover,
\begin{flalign*}
    D_\theta\Psi(t)
    &=
    \sum_{k=1}^M
    k
    \int_{[0,T]^{k-1}}
    g_k(s_1,\dots,s_{k-1},\theta,t)
    \,dB(s_1)\cdots dB(s_{k-1}),
    \\
    D_{\theta'}D_\theta\Psi(t)
    &=
    \sum_{k=2}^M
    k(k-1)
    \int_{[0,T]^{k-2}}
    g_k(s_1,\dots,s_{k-2},\theta,\theta',t)
    \,dB(s_1)\cdots dB(s_{k-2}).
\end{flalign*}
Thus, by hypercontractivity and \eqref{eq:example-1-p1-kernel},
\begin{flalign*}
    \mathbb E\bigg[
    \bigg(
    \int_0^T
    |D_\theta\Psi(t)-D_{\theta'}\Psi(t)|^2\,dt
    \bigg)^{\frac p2}
    \bigg]
    \leq
    C
    \bigg(
    \sum_{k=1}^M
    \int_0^T
    \big\|
    g_k(\cdot,\theta,t)-g_k(\cdot,\theta',t)
    \big\|_{L^2([0,T]^{k-1})}^2\,dt
    \bigg)^{\frac p2}
    \leq
    C|\theta-\theta'|^{\frac p2}.
\end{flalign*}
Similarly, \eqref{eq:example-1-p3-kernel} gives
\begin{flalign*}
    \sup_{\theta,\theta'\in[0,T]}
    \mathbb E\bigg[
    \bigg(
    \int_0^T
    |D_{\theta'}D_\theta\Psi(t)|^2\,dt
    \bigg)^{\frac p2}
    \bigg]
    \leq
    C
    \sup_{\theta,\theta'\in[0,T]}
    \bigg(
    \sum_{k=2}^M
    \int_0^T
    \big\|
    g_k(\cdot,\theta,\theta',t)
    \big\|_{L^2([0,T]^{k-2})}^2\,dt
    \bigg)^{\frac p2}
    <\infty.
\end{flalign*}
Therefore, $\mathfrak p_1,\mathfrak p_3<\infty$.

It remains to verify that $\mathfrak p_2<\infty$. For each fixed $t$, hypercontractivity and \eqref{eq:example-1-eq-1} yield
\begin{flalign*}
    \mathbb E|D_\theta\Psi(t)|^q
    &\leq
    C
    \bigg(
    \sum_{k=1}^M
    \big\|
    g_k(\cdot,\theta,t)
    \big\|_{L^2([0,T]^{k-1})}^2
    \bigg)^{\frac q2},
\end{flalign*}
and hence
\begin{gather*}
    \sup_{\theta,t\in[0,T]}
    \mathbb E|D_\theta\Psi(t)|^q
    <\infty.
\end{gather*}
Likewise, by \eqref{eq:example-1-eq-2},
\begin{flalign*}
    \sup_{\theta\in[0,T]}
    \mathbb E\big|
    D_\theta\Psi(t)-D_\theta\Psi(r)
    \big|^q
    &\leq
    C
    \bigg(
    \sum_{k=1}^M
    \big\|
    g_k(\cdot,\theta,t)-g_k(\cdot,\theta,r)
    \big\|_{L^2([0,T]^{k-1})}^2
    \bigg)^{\frac q2}
    \leq
    C|t-r|^{\eta q}.
\end{flalign*}
Since $\eta q>1$, the strengthened Kolmogorov criterion, see
\cite[Theorem 2.1]{revuz}, implies that, for every
$\gamma<\eta-\frac1q$,
\begin{gather*}
    \sup_{\theta\in[0,T]}
    \mathbb E\bigg[
    \sup_{t\neq r}
    \frac{
    |D_\theta\Psi(t)-D_\theta\Psi(r)|^q
    }{
    |t-r|^{\gamma q}
    }
    \bigg]
    <\infty.
\end{gather*}
Consequently,
\begin{flalign*}
    \mathfrak p_2^q
    &=
    \sup_{\theta\in[0,T]}
    \mathbb E\bigg[
    \sup_{t\in[0,T]}
    |D_\theta\Psi(t)|^q
    \bigg]
    \leq
    C
    \sup_{\theta\in[0,T]}
    \mathbb E|D_\theta\Psi(0)|^q
    +
    CT^{\gamma q}
    \sup_{\theta\in[0,T]}
    \mathbb E\bigg[
    \sup_{t\neq r}
    \frac{
    |D_\theta\Psi(t)-D_\theta\Psi(r)|^q
    }{
    |t-r|^{\gamma q}
    }
    \bigg]
    <\infty.
\end{flalign*}
Hence all the conditions of Assumption \ref{assumption-Holder-BSVIE} are satisfied.

\subsection*{Example 2}

Under the assumptions on $b$ and $\sigma$, we have
$X(t)\in\mathbb D^{2,m}$ for every $m\geq2$ and $t\in[0,T]$. Moreover,
for every $\theta,\theta'\in[0,T]$,
\begin{gather}\label{eq:moment-estimates-X}
    \mathbb E\bigg[
    \sup_{t\in[0,T]}|X(t)|^m
    \bigg]
    +
    \mathbb E\bigg[
    \sup_{t\in[\theta,T]}
    |D_\theta X(t)|^m
    \bigg]
    +
    \mathbb E\bigg[
    \sup_{t\in[\theta\vee\theta',T]}
    |D_{\theta'}D_\theta X(t)|^m
    \bigg]
    \leq
    C_m.
\end{gather}
We also have
\begin{gather*}
    \mathbb E
    |D_\theta X(t)-D_\theta X(s)|^m
    \leq
    C_m|t-s|^{\frac m2},
    \qquad
    s,t\in[\theta,T],
\end{gather*}
and
\begin{gather}\label{eq:theta-regularity-DX}
    \int_{\theta\vee\theta'}^T
    \mathbb E
    \big|
    D_\theta X(t)-D_{\theta'}X(t)
    \big|^m\,dt
    \leq
    C_m|\theta-\theta'|^{\frac m2}.
\end{gather}
In what follows, the moment exponent $m$ is chosen sufficiently large,
depending on $p$, $q$, and the polynomial-growth exponents.

We first verify the conditions on the terminal datum
\[
    \Psi(t)
    =
    \psi\big(t,\varphi(t),\varphi(T)\big).
\]
The polynomial growth of $\psi$ and $h$, together with
\eqref{eq:moment-estimates-X}, gives
$\|\Psi\|_{2,p}<\infty$.

By the Malliavin chain rule for functions with derivatives of polynomial
growth, see \cite[Proposition 5.4]{Nualart2019},
\begin{gather*}
    D_\theta\varphi(t)
    =
    \mathbf 1_{\{\theta\leq t\}}
    \int_\theta^t
    \partial_xh(r,X(r))D_\theta X(r)\,dr.
\end{gather*}
We claim that, for every $m\geq2$,
\begin{flalign}\label{eq:p1-varphi-sketch}
    &\mathbb E\bigg[
    \bigg(
    \int_0^T
    |D_\theta\varphi(t)-D_{\theta'}\varphi(t)|^2\,dt
    \bigg)^{\frac m2}
    \bigg]
    +
    \mathbb E\big[
    |D_\theta\varphi(T)-D_{\theta'}\varphi(T)|^m
    \big]
    \leq
    C_m|\theta-\theta'|^{\frac m2}.
\end{flalign}
Indeed, assuming $\theta<\theta'$, for $t\geq\theta'$,
\begin{flalign*}
    D_\theta\varphi(t)-D_{\theta'}\varphi(t)
    &=
    \int_\theta^{\theta'}
    \partial_xh(r,X(r))D_\theta X(r)\,dr
    +
    \int_{\theta'}^t
    \partial_xh(r,X(r))
    \big(
    D_\theta X(r)-D_{\theta'}X(r)
    \big)\,dr,
\end{flalign*}
whereas $D_{\theta'}\varphi(t)=0$ for $t\in[\theta,\theta')$.
Estimate \eqref{eq:p1-varphi-sketch} follows from Hölder's inequality,
the polynomial growth of $\partial_xh$, \eqref{eq:moment-estimates-X},
and \eqref{eq:theta-regularity-DX}.

Furthermore,
\begin{flalign}\label{eq:example-m-chain-rule-Psi}
    D_\theta\Psi(t)
    =
    \partial_1\psi\big(t,\varphi(t),\varphi(T)\big)
    D_\theta\varphi(t)
    +
    \partial_2\psi\big(t,\varphi(t),\varphi(T)\big)
    D_\theta\varphi(T).
\end{flalign}
Consequently, choosing $m>p$ sufficiently large and using the polynomial
growth of the derivatives of $\psi$, Hölder's inequality, and
\eqref{eq:p1-varphi-sketch}, we obtain
\begin{gather*}
    \mathbb E\bigg[
    \bigg(
    \int_0^T
    |D_\theta\Psi(t)-D_{\theta'}\Psi(t)|^2\,dt
    \bigg)^{\frac p2}
    \bigg]
    \leq
    C|\theta-\theta'|^{\frac p2}.
\end{gather*}
Thus $\mathfrak p_1<\infty$.

The second-order chain rule gives
\begin{flalign*}
    D_{\theta'}D_\theta\varphi(t)
    =
    \mathbf 1_{\{\theta\vee\theta'\leq t\}}
    \int_{\theta\vee\theta'}^t
    \bigg[
    \partial_{xx}h(r,X(r))
    D_\theta X(r)D_{\theta'}X(r)
    +
    \partial_xh(r,X(r))
    D_{\theta'}D_\theta X(r)
    \bigg]\,dr.
\end{flalign*}
Hence the polynomial growth of the derivatives of $h$ and $\psi$,
together with \eqref{eq:moment-estimates-X}, implies
\begin{gather*}
    \sup_{\theta,\theta'\in[0,T]}
    \mathbb E\bigg[
    \bigg(
    \int_0^T
    |D_{\theta'}D_\theta\Psi(t)|^2\,dt
    \bigg)^{\frac p2}
    \bigg]
    <\infty.
\end{gather*}
Therefore $\mathfrak p_3<\infty$.

To verify $\mathfrak p_2<\infty$, observe that, for every $m\geq2$,
\begin{gather}\label{eq:moment-varphi-Dvarphi}
    \mathbb E\bigg[
    \sup_{t\in[0,T]}|\varphi(t)|^m
    \bigg]
    +
    \sup_{\theta\in[0,T]}
    \mathbb E\bigg[
    \sup_{t\in[0,T]}|D_\theta\varphi(t)|^m
    \bigg]
    <\infty.
\end{gather}
If $\ell$ denotes a sufficiently large polynomial-growth exponent for
the derivatives of $\psi$, then \eqref{eq:example-m-chain-rule-Psi}
gives
\begin{gather*}
    \sup_{t\in[0,T]}|D_\theta\Psi(t)|
    \leq
    C
    \bigg(
    1+\sup_{t\in[0,T]}|\varphi(t)|
    \bigg)^\ell
    \sup_{t\in[0,T]}|D_\theta\varphi(t)|.
\end{gather*}
Therefore, Hölder's inequality and
\eqref{eq:moment-varphi-Dvarphi} yield
\begin{gather*}
    \mathfrak p_2^q
    =
    \sup_{\theta\in[0,T]}
    \mathbb E\bigg[
    \sup_{t\in[0,T]}
    |D_\theta\Psi(t)|^q
    \bigg]
    <\infty.
\end{gather*}
Thus all the conditions on the terminal datum are satisfied.

We next consider
\[
    G(t,s,y,z)
    \coloneqq
    g(t,s,X(t),X(s),y,z).
\]
The assumptions on $g$ give the required regularity with respect to
$(y,z)$. Moreover,
\begin{gather*}
    |G(t,s,0,0)|
    \leq
    |g(t,s,0,0,0,0)|
    +
    C\big(
    |X(t)|+|X(s)|
    \big),
\end{gather*}
so that \eqref{eq:moment-estimates-X} and the integrability assumption on
$g(t,s,0,0,0,0)$ imply $\norm{G^0}_{\mathcal H_\Delta^q}<\infty$.

Extending the Malliavin derivatives of $X$ by zero before their respective
Malliavin times, the chain rule gives
\begin{flalign*}
    D_\theta G(t,s,y,z)
    &=
    \partial_{x_1}g(t,s,X(t),X(s),y,z)D_\theta X(t)
    +
    \partial_{x_2}g(t,s,X(t),X(s),y,z)D_\theta X(s).
\end{flalign*}
The same identity, with the corresponding mixed partial derivatives of
$g$, holds for $D_\theta\partial_yG$ and
$D_\theta\partial_zG$. Hence we may take
\begin{gather*}
    \mathcal G_\theta^1(t,s)
    =
    \mathcal G_\theta^y(t,s)
    =
    \mathcal G_\theta^z(t,s)
    =
    C\big(
    |D_\theta X(t)|+|D_\theta X(s)|
    \big).
\end{gather*}
The definitions of $\mathcal T_\Delta^q$, $\mathcal S_\Delta^q$, and
$\mathcal U_\Delta^q$, together with
\eqref{eq:moment-estimates-X}, then give
\begin{gather*}
    \gamma_1,\gamma_y,\gamma_z<\infty.
\end{gather*}
The required Malliavin differentiability and continuity properties follow
from the chain rule, the continuity of the derivatives of $g$, and the
moment estimates for $X$, $DX$, and $D^2X$.

It remains to verify the regularity in the Malliavin parameter. Assume
$\theta<\theta'$ and set
\begin{gather*}
    H_{\theta,\theta'}(u)
    \coloneqq
    |D_\theta X(u)-D_{\theta'}X(u)|,
\end{gather*}
where the derivatives are again extended by zero before their Malliavin times. The boundedness of $\partial_{x_1}g$ and $\partial_{x_2}g$ gives
\begin{gather*}
    |D_\theta G(t,s,y,z)-D_{\theta'}G(t,s,y,z)|
    \leq
    \mathcal K_{\theta,\theta'}(t,s), \quad \text{where} \quad \mathcal K_{\theta,\theta'}(t,s)
    \coloneqq
    C\big(
    H_{\theta,\theta'}(t)
    +
    H_{\theta,\theta'}(s)
    \big).
\end{gather*}
Moreover,
\begin{gather*}
    \int_0^T\int_t^T
    |\mathcal K_{\theta,\theta'}(t,s)|^2\,ds\,dt
    \leq
    C\int_0^T
    |H_{\theta,\theta'}(u)|^2\,du.
\end{gather*}
Since
\begin{flalign*}
    H_{\theta,\theta'}(u)
    &=
    \mathbf 1_{[\theta,\theta')}(u)|D_\theta X(u)|
    +
    \mathbf 1_{[\theta',T]}(u)
    |D_\theta X(u)-D_{\theta'}X(u)|,
\end{flalign*}
Hölder's inequality, \eqref{eq:moment-estimates-X}, and
\eqref{eq:theta-regularity-DX} with $m=p$ yield
\begin{flalign*}
    \mathbb E\bigg[
    \bigg(
    \int_0^T
    |H_{\theta,\theta'}(u)|^2\,du
    \bigg)^{\frac p2}
    \bigg]
    \leq
    C|\theta-\theta'|^{\frac p2}
    +
    C
    \int_{\theta'}^T
    \mathbb E\big[
    |D_\theta X(u)-D_{\theta'}X(u)|^p
    \big]\,du
\leq
    C|\theta-\theta'|^{\frac p2}.
\end{flalign*}
Consequently, $\kappa<\infty$.

Finally, the second-order chain rule gives
\begin{flalign*}
    |D_{\theta'}D_\theta G(t,s,y,z)|
    &\leq
    C\bigg(
    |D_{\theta'}D_\theta X(t)|
    +
    |D_{\theta'}D_\theta X(s)|
    +
    \big(
    |D_\theta X(t)|+|D_\theta X(s)|
    \big)
    \big(
    |D_{\theta'}X(t)|+|D_{\theta'}X(s)|
    \big)
    \bigg).
\end{flalign*}
We may therefore take the right-hand side as
$\mathcal G_{\theta,\theta'}^2(t,s)$. By
\eqref{eq:moment-estimates-X}, Hölder's inequality, and the fact that the
moments of $X$, $DX$, and $D^2X$ are finite at every order,
\begin{gather*}
    \gamma_2
    =
    \sup_{\theta,\theta'\in[0,T]}
    \norm{\mathcal G_{\theta,\theta'}^2}_{\mathcal H_\Delta^p}
    <\infty.
\end{gather*}
This completes the verification of Assumption
\ref{assumption-Holder-BSVIE}.

\section{Proof of Theorem \ref{thm:euler-error-bsde-system}}
\label{section:proof-euler}

For ease of notation, set
\begin{flalign*}
    \delta^\pi Y(t_k,s)
    &\coloneqq
    Y^\pi(t_k,s)-\mathcal Y^\pi(t_k,s),
    \qquad
    \delta^\pi Z(t_k,s)
    \coloneqq
    Z^\pi(t_k,s)-\mathcal Z^\pi(t_k,s),
\end{flalign*}
and
\begin{gather*}
    y_{k,l}
    \coloneqq
    \mathbb E
    \big|
    \delta^\pi Y(t_k,t_l)
    \big|^2,
    \qquad
    z_{k,l}
    \coloneqq
    \mathbb E
    \int_{t_l}^{t_{l+1}}
    \big|
    \delta^\pi Z(t_k,s)
    \big|^2\,ds,
    \qquad
    I_{k,l}
    \coloneqq
    y_{k,l}+z_{k,l}.
\end{gather*}
We also write
\begin{flalign*}
    \mathcal O^Z(\pi)
    \coloneqq
    \rho(|\pi|)^2+|\pi|
    +
    \sum_{k=0}^{N-1}
    \Delta_k
    \mathbb E
    \int_{t_k}^{t_{k+1}}
    |Z^\pi(t_k,s)|^2\,ds
    +
    \sum_{k=0}^{N-2}
    \Delta_k
    \sum_{l=k+1}^{N-2}
    \mathbb E
    \int_{t_l}^{t_{l+1}}
    \big|
    Z^\pi(t_k,s)-\scalar{Z^\pi}_{k,l+1}
    \big|^2\,ds.
\end{flalign*}
If $\tau^\star(t)$ denotes the right endpoint of the cell containing $t$,
then
\begin{flalign}
    \mathcal E^Z(\pi)
    &=
    \mathbb E
    \int_0^T
    \int_t^{\tau^\star(t)}
    |\delta^\pi Z(\tau(t),s)|^2\,ds\,dt
    +
    \sum_{k=0}^{N-2}
    \Delta_k
    \sum_{l=k+1}^{N-1}
    z_{k,l}.
    \label{eq:EZ-diagonal-offdiagonal-decomposition}
\end{flalign}

\medskip
\noindent
\textit{Step 1.}
Let $0\leq k<l\leq N-1$. Subtracting the equations for
$Y^\pi$ and $\mathcal Y^\pi$ on $[t_l,t_{l+1}]$ gives
\begin{flalign*}
    \delta^\pi Y(t_k,t_l)
    &=
    \delta^\pi Y(t_k,t_{l+1})
    +
    \int_{t_l}^{t_{l+1}}
    R^\pi_{k,l}(s)\,ds
    -
    \int_{t_l}^{t_{l+1}}
    \delta^\pi Z(t_k,s)\,dB(s),
\end{flalign*}
where
\begin{flalign*}
    R^\pi_{k,l}(s)
    &\coloneqq
    G_k^\pi
    \big(
    s,Y^\pi(t_l,s),Z^\pi(t_k,s)
    \big)
    -
    G_k^\pi
    \big(
    t_{l+1},
    \mathcal Y^\pi(t_l,t_{l+1}),
    \scalar{\mathcal Z^\pi}_{k,l+1}
    \big).
\end{flalign*}
Hence, by Itô's isometry,
\begin{flalign*}
    I_{k,l}
    &=
    \mathbb E
    \bigg|
    \delta^\pi Y(t_k,t_{l+1})
    +
    \int_{t_l}^{t_{l+1}}
    R^\pi_{k,l}(s)\,ds
    \bigg|^2.
\end{flalign*}
For $l\leq N-2$, Young's inequality gives, for every $\varepsilon>0$,
\begin{flalign*}
    I_{k,l}
    &\leq
    \bigg(
    1+\frac{\Delta_l}{\varepsilon}
    \bigg)y_{k,l+1}
    +
    (\Delta_l+\varepsilon)
    \mathbb E
    \int_{t_l}^{t_{l+1}}
    |R^\pi_{k,l}(s)|^2\,ds.
\end{flalign*}
Define
\begin{flalign*}
    R^{0,\pi}_{k,l}(s)
    &\coloneqq
    G_k^\pi
    \big(
    s,Y^\pi(t_l,s),Z^\pi(t_k,s)
    \big)
    -
    G_k^\pi
    \big(
    t_{l+1},Y^\pi(t_l,s),Z^\pi(t_k,s)
    \big).
\end{flalign*}
Using the Lipschitz property, inserting
$\scalar{Z^\pi}_{k,l+1}$, and applying conditional Jensen's inequality
and the mesh-ratio condition, we obtain
\begin{flalign*}
    \mathbb E
    \int_{t_l}^{t_{l+1}}
    |R^\pi_{k,l}(s)|^2\,ds
    &\leq
    C\Delta_l y_{l,l+1}
    +
    C_1z_{k,l+1}
    +
    \Gamma_{k,l},
\end{flalign*}
where
\begin{flalign*}
    \Gamma_{k,l}
    &\coloneqq
    C
    \mathbb E
    \int_{t_l}^{t_{l+1}}
    |R^{0,\pi}_{k,l}(s)|^2\,ds
    +
    C
    \mathbb E
    \int_{t_l}^{t_{l+1}}
    |Y^\pi(t_l,s)-Y^\pi(t_l,t_{l+1})|^2\,ds
    +
    C
    \mathbb E
    \int_{t_l}^{t_{l+1}}
    \big|
    Z^\pi(t_k,s)-\scalar{Z^\pi}_{k,l+1}
    \big|^2\,ds.
\end{flalign*}
Choosing $\varepsilon=(4C_1)^{-1}$, using
$y_{k,l+1}=I_{k,l+1}-z_{k,l+1}$, and taking $|\pi|$ sufficiently small,
we obtain
\begin{flalign}
    I_{k,l}
    +
    \frac12 z_{k,l+1}
    &\leq
    (1+C\Delta_l)I_{k,l+1}
    +
    C\Delta_l I_{l,l+1}
    +
    \Gamma_{k,l}.
    \label{eq:euler-local-recursion}
\end{flalign}
On the last cell, $\delta^\pi Y(t_k,t_N)=0$ and
$\scalar{\mathcal Z^\pi}_{k,N}=0$. The same argument yields
\begin{flalign}
    I_{k,N-1}
    &\leq
    \Gamma_{k,N-1},
    \qquad
    0\leq k\leq N-2,
    \label{eq:euler-terminal-recursion}
\end{flalign}
where
\begin{flalign*}
    \Gamma_{k,N-1}
    \coloneqq
    C\Delta_{N-1}
    \mathbb E
    \int_{t_{N-1}}^{t_N}
    |R^{0,\pi}_{k,N-1}(s)|^2\,ds
    +
    C\Delta_{N-1}
    \mathbb E
    \int_{t_{N-1}}^{t_N}
    |Y^\pi(t_{N-1},s)-Y^\pi(t_{N-1},t_N)|^2\,ds
    +
    C\Delta_{N-1}
    \mathbb E
    \int_{t_{N-1}}^{t_N}
    |Z^\pi(t_k,s)|^2\,ds.
\end{flalign*}

\medskip
\noindent
\textit{Step 2.}
For $0\leq k\leq N-2$, set
\begin{gather*}
    H_k
    \coloneqq
    \max_{l\in\{k+1,\dots,N-1\}}
    I_{k,l}.
\end{gather*}
Dropping the non-negative $z$-term in
\eqref{eq:euler-local-recursion}, applying the discrete Gronwall
inequality in $l$, and using \eqref{eq:euler-terminal-recursion}, we get
\begin{flalign*}
    H_k
    &\leq
    C
    \sum_{l=k+1}^{N-1}
    \Gamma_{k,l}
    +
    C
    \sum_{l=k+1}^{N-2}
    \Delta_l H_l.
\end{flalign*}
A second discrete Gronwall argument, followed by summation in $k$, gives
\begin{flalign}
    \sum_{k=0}^{N-2}
    \Delta_k H_k
    &\leq
    C
    \sum_{k=0}^{N-2}
    \Delta_k
    \sum_{l=k+1}^{N-1}
    \Gamma_{k,l}.
    \label{eq:euler-H-bound}
\end{flalign}

Summing \eqref{eq:euler-local-recursion} in $l$, using
\eqref{eq:euler-terminal-recursion}, and then
\eqref{eq:euler-H-bound}, we similarly obtain
\begin{flalign}
    \sum_{k=0}^{N-2}
    \Delta_k H_k
    +
    \sum_{k=0}^{N-2}
    \Delta_k
    \sum_{l=k+1}^{N-1}
    z_{k,l}
    \leq
    C
    \sum_{k=0}^{N-2}
    \Delta_k
    \sum_{l=k+1}^{N-1}
    \Gamma_{k,l}.
    \label{eq:euler-errors-by-Gamma}
\end{flalign}

It remains to estimate the right-hand side. Assumption
\ref{assumption:continuity-nodal-data} and the a priori estimates for the
BSDE system yield
\begin{gather*}
    \sum_{k=0}^{N-2}
    \Delta_k
    \sum_{l=k+1}^{N-2}
    \mathbb E
    \int_{t_l}^{t_{l+1}}
    |R^{0,\pi}_{k,l}(s)|^2\,ds
    \leq
    C\rho(|\pi|)^2.
\end{gather*}
Moreover, since the $l$th generator vanishes on its diagonal cell,
\begin{gather*}
    Y^\pi(t_l,s)-Y^\pi(t_l,t_{l+1})
    =
    -
    \int_s^{t_{l+1}}
    Z^\pi(t_l,r)\,dB(r),
\end{gather*}
and hence
\begin{flalign*}
    \mathbb E
    \int_{t_l}^{t_{l+1}}
    |Y^\pi(t_l,s)-Y^\pi(t_l,t_{l+1})|^2\,ds
    &\leq
    \Delta_l
    \mathbb E
    \int_{t_l}^{t_{l+1}}
    |Z^\pi(t_l,r)|^2\,dr.
\end{flalign*}
The remaining term in $\Gamma_{k,l}$ is precisely the oscillation term in
$\mathcal O^Z(\pi)$. The last-cell contributions are estimated in the same
way, while the a priori estimate gives
\begin{gather*}
    \sum_{k=0}^{N-2}
    \Delta_k\Delta_{N-1}
    \mathbb E
    \int_{t_{N-1}}^{t_N}
    |Z^\pi(t_k,s)|^2\,ds
    \leq
    C|\pi|.
\end{gather*}
Therefore,
\begin{flalign}
    \sum_{k=0}^{N-2}
    \Delta_k
    \sum_{l=k+1}^{N-1}
    \Gamma_{k,l}
    &\leq
    C\mathcal O^Z(\pi).
    \label{eq:euler-Gamma-bound}
\end{flalign}
Combining \eqref{eq:euler-errors-by-Gamma} and
\eqref{eq:euler-Gamma-bound}, we obtain
\begin{flalign}
    \sum_{k=0}^{N-2}
    \Delta_k H_k
    +
    \sum_{k=0}^{N-2}
    \Delta_k
    \sum_{l=k+1}^{N-1}
    z_{k,l}
    &\leq
    C\mathcal O^Z(\pi).
    \label{eq:euler-offdiagonal-bound}
\end{flalign}

\medskip
\noindent
\textit{Step 3.}
For $s\in[t_k,t_{k+1}]$, both generators vanish and therefore
\begin{gather*}
    \delta^\pi Y(t_k,s)
    =
    \delta^\pi Y(t_k,t_{k+1})
    -
    \int_s^{t_{k+1}}
    \delta^\pi Z(t_k,r)\,dB(r).
\end{gather*}
Conditional expectation and Itô's isometry imply
\begin{gather*}
    \mathbb E
    |\delta^\pi Y(t_k,s)|^2
    \leq
    y_{k,k+1},
    \qquad
    \mathbb E
    \int_{t_k}^{t_{k+1}}
    |\delta^\pi Z(t_k,s)|^2\,ds
    \leq
    y_{k,k+1}.
\end{gather*}
Hence, for $k\leq N-2$,
\begin{flalign*}
    \mathbb E
    \int_{t_k}^{t_{k+1}}
    |\delta^\pi Y(t_k,s)|^2\,ds
    +
    \mathbb E
    \int_{t_k}^{t_{k+1}}
    \int_t^{t_{k+1}}
    |\delta^\pi Z(t_k,s)|^2\,ds\,dt
    \leq
    C\Delta_kH_k.
\end{flalign*}
For $k=N-1$, both errors vanish because the exact and discrete equations
coincide on $[t_{N-1},T]$. Consequently,
\begin{flalign*}
    \mathcal E^Y(\pi)
    +
    \mathbb E
    \int_0^T
    \int_t^{\tau^\star(t)}
    |\delta^\pi Z(\tau(t),s)|^2\,ds\,dt
    &\leq
    C
    \sum_{k=0}^{N-2}
    \Delta_kH_k
    \leq
    C\mathcal O^Z(\pi).
\end{flalign*}
Combining this estimate with
\eqref{eq:EZ-diagonal-offdiagonal-decomposition} and
\eqref{eq:euler-offdiagonal-bound} proves the result.

\bibliographystyle{abbrvnat}
\bibliography{bibliography}

\end{document}